\documentclass[12pt]{amsart}
\usepackage{amssymb,amsmath,amsthm,verbatim,hyperref,enumerate,mathtools}
\usepackage{tikz}
\usetikzlibrary{positioning}

\theoremstyle{plain}
\newtheorem{thm}{Theorem}[section]
\newtheorem{lem}[thm]{Lemma}

\newtheorem{cor}[thm]{Corollary}
\newtheorem{obs}[thm]{Observation}
\newtheorem{prop}[thm]{Proposition}
\theoremstyle{definition}
\newtheorem{defn}[thm]{Definition}
\newtheorem{remark}[thm]{Remark}
\newtheorem{ex}[thm]{Example}
\newtheorem{rk}[thm]{Remark}
\newtheorem{qn}[thm]{Question}

\newcommand\I{{\mathcal{I}}} 
\renewcommand\S{{\mathcal{S}}}
\renewcommand\L{{\mathcal{L}}}
\newcommand\D{{\mathcal{D}}}
\DeclareMathOperator{\er}{er}
\DeclareMathOperator{\fp}{fp}
\newcommand\bh{{\overline{h}}}
\newcommand{\lex}{<_{\mathrm{lex}}}

\newcommand{\T}{{\mathcal{T}}}
\newcommand{\card}[1]{\left|#1\right|}
\DeclareMathOperator{\shift}{shift}
\DeclareMathOperator{\sgn}{sgn}
\DeclareMathOperator{\st}{st}

\def\N{-- ++(0,1)}
\def\E{-- ++(1,0)}
\newcommand{\cross}[2]{ \draw[thick] (#1,#2)--(#1-1,#2-1);
     \draw[thick] (#1-1,#2)--(#1,#2-1);}
\newcommand{\perm}[1]{
 \foreach \y [count=\x] in #1 {
     \cross{\x}{\y}
}}

\newcommand{\delete}[1]{}

\title{Complexes of pattern-avoiding injective words}

\author{Sergi Elizalde}
\address{Department of Mathematics, Dartmouth College, Hanover, NH 03755}
\email{sergi.elizalde@dartmouth.edu}

\author{Philip Hanlon}
\address{Department of Mathematics, Dartmouth College, Hanover, NH 03755}
\email{philip.j.hanlon@dartmouth.edu}

\author{Patricia Hersh}
\address{Department of Mathematics, University of Oregon, Eugene, OR 97403}
\email{plhersh@uoregon.edu}

\begin{document}

\begin{abstract}
The complex of injective words is a cell complex that arises in a number of different areas. It has applications to proving homological stability and to the study of group cohomology, and it is closely related to the random-to-random  Markov chain.  
This complex was 
first 
studied by Farmer, who proved it has the homology of a wedge of  
top-dimensional spheres. Later, Bj\"orner and Wachs established its shellability, and Reiner and Webb uncovered its $\S_n$-module structure, observing in the process that the rank of its top homology group is the $n$th derangement number.

We introduce natural subcomplexes of the complex of injective words  by fixing a permutation pattern $\sigma$ and considering only those injective  words in the alphabet $\{1,2,\dots,n\} $ that avoid $\sigma$. We prove that such pattern-avoiding complexes are shellable if $\sigma$ begins or ends with its largest or smallest letter, and we construct homology bases for the  complexes avoiding such patterns.
For patterns of length 3, all of which have this property, we show that the rank of the top homology of the resulting complex is a Riordan number. All but four patterns of length 4 also  have this property, and  for two of the remaining four patterns, we establish shellability using a different method.

We also introduce a technique to use enumerative combinatorics to prove shellability, and we apply it to the complex of separable injective words, thereby deducing shellability in this case.    Along the way, we give a combinatorial formula for all of the $h$-numbers in the full complex of injective words, as well as for each of  the  subcomplexes which we prove are shellable.  
Going in the other direction, we use shellability of complexes of pattern-avoiding injective words to
deduce new refined counting formulas for pattern-avoiding permutations.
\end{abstract}


\maketitle

\section{Introduction}

The complex of injective words, denoted by $\I_n$, is a cell complex of Boolean type (a slight generalization of the notion of simplicial complex, see Section~\ref{Boolean-section} for definitions) whose faces are all the ``injective words'' over the alphabet $[n]\coloneqq\{1,2,\dots,n\}$, that is, words with no repeated letters.  
A face of $\I_n$ is in the closure of another face  if and only if the former injective word is a subword of the latter injective word. We use the term {\em subword} to refer to any subsequence, not necessarily consecutive. For example, in $\I_5$, the face $351$ is in the closure of the face $32514$, but not in the closure of the face $4531$.
In particular, the  maximal faces (known as facets) of $\I_n$ correspond to permutations of $[n]$, namely elements of the symmetric group $\S_n$.
Multiple faces can have the same set of vertices; for example, the edges $12$ and $21$ in $\I_n$ for any $n\ge 2$ each have the same two vertices, $1$ and $2$, as their endpoints.

Complexes of injective words have arisen in a number of different settings.   Reiner and Webb used them in  \cite{WebbR} to develop a general method for calculating  group cohomology in terms of subgroups.  Miller and Wilson used them in a spectral sequence argument in \cite{MW} establishing higher representation stability for the rational homology of configuration spaces of ordered points on noncompact manifolds. 
Complexes of injective words  have also appeared in work on stability in other contexts by Bestvina \cite[Proposition 6]{Best} and by Randal-Williams \cite[Proposition 3.2]{R-W}.  Their topological structure has been further studied by Chacholski, Levi and Meshulam in \cite{CLM}. 
Boyd and Hepworth developed Temperley-Lieb algebra (and Iwahori-Hecke algebra) analogues of complexes of injective words and used them for studying  stability  in these algebras in  \cite{BH}, \cite{BH2} and \cite{Hep}.  Jonsson and Welker developed a different generalization of the complex of injective words in \cite{JW}, motivated by examples of these more general complexes arising in algebraic K-theory.  Hanlon and Hersh related the discrete Laplacian for the complex of injective words   to the random to random shuffle operator, a well-known Markov chain, in \cite{Hanlon-Hersh2}; they proved that the former has integer spectrum (after suitable scaling) if and only if the latter does, and then later \cite{DS} proved that the (suitably scaled) random to random shuffle operator indeed has integer spectrum.

Farmer proved in~\cite{Fa} that $\I_n$ has the homology of a wedge of top-dimensional spheres.
It was subsequently proven by Bj\"orner and Wachs in \cite{BW}  that the dual to the face poset for $\I_n$ is lexicographically shellable (and graded), implying  
homotopy equivalence of $\I_n$ to a wedge of top dimensional spheres. An  expression was provided in \cite{BW} for the number of spheres in this wedge of spheres, answering a question raised by Farmer by giving a considerably simpler expression than  previously appeared in \cite{Fa}. Reiner and Webb showed in \cite{WebbR} that the absolute value of the reduced Euler characteristic of $\I_n$ (and hence the number of spheres) is a derangement number $d_n$. They did this by counting the number of faces of each dimension and expressing the reduced Euler characteristic as the alternating sum
$$\sum_{i=0}^n (-1)^{i-1}\binom{n}{i}i!.$$ 
In Section~\ref{sec:In} we will give a very different  proof of this fact, using shellability and a bijection of Foata to show that the top $h$-number is $d_n$. Reiner and Webb in fact  gave a representation-theoretic enrichment of this alternating sum expression for the derangement number 
in \cite{WebbR}.   
They also provided a representation theoretic explanation 
for a  symmetric function proof due to Desarmenien and Wachs \cite{DW} of the classical 
recurrence  $d_n = nd_{n-1} + (-1)^n$.
A Hodge type decomposition for the complex of injective words was given by Hanlon and Hersh in \cite{Hanlon-Hersh2}.  A homology basis for the complex of injective words was   constructed by Miller and Wilson  in \cite{MW}. 
 The $\S_n$-module structure for the rank-selected homology was studied by Athanasiadis in~\cite{Ath}.  

In this paper we introduce a natural subcomplex of $\I_n$, consisting of the injective words that avoid a given pattern $\sigma\in\S_k$, in the standard sense of having no subword whose elements are in the same relative order as those of $\sigma$.
We denote this subcomplex by $\I^\sigma_n$.
The face poset of $\I^{312}_3$ is drawn in Figure~\ref{fig:312poset}.
If $\sigma$ is a pattern of length 2, then $\I_n^\sigma$ is just a simplex. 
Thus, the rest of the paper will focus on patterns of length at least~3.

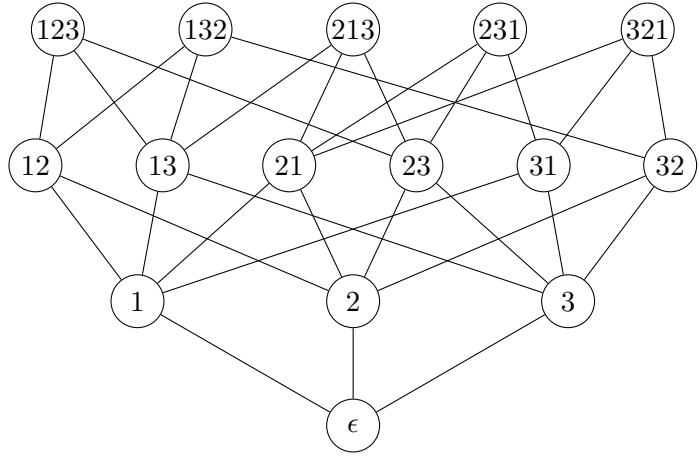
\begin{figure}[htb]
    \begin{tikzpicture}[scale=.9,
  every node/.style={font=\small,circle,draw,minimum size=7mm,inner sep=0pt},
]

\node (e) at (5.67,0) {$\epsilon$};
\node (a1) at (2.50,1.83) {1};
\node (a2) at (5.67,1.83) {2};
\node (a3) at (8.83,1.83) {3};
\node (b12) at ( 1.00,3.83) {12};
\node (b13) at ( 2.87,3.83) {13};
\node (b21) at ( 4.73,3.83) {21};
\node (b23) at ( 6.60,3.83) {23};
\node (b31) at ( 8.47,3.83) {31};
\node (b32) at (10.33,3.83) {32};
\node (c123) at ( 1.33,5.83) {123};
\node (c132) at ( 3.50,5.83) {132};
\node (c213) at ( 5.67,5.83) {213};
\node (c231) at ( 7.83,5.83) {231};
\node (c321) at (10.00,5.83) {321};

\draw (e) -- (a1);
\draw (e) -- (a2);
\draw (e) -- (a3);

\draw (a1) -- (b12);
\draw (a1) -- (b13);
\draw (a1) -- (b21);
\draw (a1) -- (b31);
\draw (a2) -- (b12);
\draw (a2) -- (b21);
\draw (a2) -- (b23);
\draw (a2) -- (b32);
\draw (a3) -- (b13);
\draw (a3) -- (b31);
\draw (a3) -- (b23);
\draw (a3) -- (b32);

\draw (c123) -- (b12);
\draw (c123) -- (b13);
\draw (c123) -- (b23);
\draw (c132) -- (b12);
\draw (c132) -- (b13);
\draw (c132) -- (b32);
\draw (c213) -- (b13);
\draw (c213) -- (b21);
\draw (c213) -- (b23);
\draw (c231) -- (b21);
\draw (c231) -- (b23);
\draw (c231) -- (b31);
\draw (c321) -- (b21);
\draw (c321) -- (b31);
\draw (c321) -- (b32);
\end{tikzpicture}
\label{fig:312poset}
\caption{The face poset of the complex $\I^{312}_3$.}
\end{figure}

For patterns $\sigma$ of length 3, we show that the reduced Euler characteristic of $\I_n^\sigma$, up to sign, is again an interesting number, namely the Riordan number $R_n$. 
Comparing this with the  expression for the reduced Euler characteristic as an alternating sum of face counts yields the seemingly new identity
\begin{equation}\label{eq:Riordan-Catalan}
   \sum_{i=0}^n (-1)^{n-i}\binom{n}{i}C_{i} = R_n,
\end{equation}    
where $C_{i}$ denotes the $i$th Catalan number.  We will discuss this identity in more depth later in the introduction, as an illustration of the interplay of topological and enumerative methods prevalent throughout this paper.  While the left side of this identity is rather straightforward, our justification for  the right side involves shellability theory, bijective combinatorics and  generating function techniques.

Now to the precise definition of $\I_n^\sigma$.  Given an injective word $\pi=\pi_1\pi_2\dots\pi_r$ and a pattern $\sigma\in \S_k$, we say that $\pi$ {\em contains} $\sigma$ if there exist indices $i_1<i_2<\dots<i_k$ such that the subword $\pi_{i_1}\pi_{i_2}\dots\pi_{i_k}$ is order-isomorphic to $\sigma_1\sigma_2\dots\sigma_k$, meaning that the entries are in the same relative order. If $\pi$ does not contain $\sigma$, we say that $\pi$ {\em avoids} $\sigma$, or that $\pi$ is {\em $\sigma$-avoiding}. 
Note that this notion of pattern containment is different from the one that we used to determine if a face of $\I_n$ is contained in the closure of another, where we considered the exact values of the entries of the subword and not only their relative order.
More generally, if $\Sigma$ is any collection of patterns (not necessarily of the same size), we say that $\pi$ avoids $\Sigma$ if it avoids every pattern in~$\Sigma$.
We will assume in the rest of the paper that all patterns have length at least $2$. This is to avoid the degenerate case where the only permutation avoiding $\sigma$ is the empty permutation.

The next simple proposition shows that the $\sigma$-avoiding injective words over the alphabet $[n]$ form a subcomplex of $\I_n$. 

\begin{prop}\label{subcomplex-property} 
For any pattern $\sigma\in\S_k$, 
the set of $\sigma$-avoiding injective words over the alphabet $[n]$ forms a subcomplex $\I^\sigma_n$ 
of the complex $\I_n$ of injective words.  More generally, for any collection $\Sigma$ of patterns, the set of injective words over the alphabet $[n]$ avoiding $\Sigma$ forms a subcomplex $\I^\Sigma_n$  of~$\I_n$.
\end{prop}

\begin{proof}
It is clear by definition  that avoiding a given  pattern  (or collection of patterns) is a property that is  preserved under taking subwords.  
\end{proof}

We will prove that $\I_n^{\sigma }$ is pure of dimension $n-1$ for large classes of patterns $\sigma $,  including all patterns of length at most 4 and all patterns starting with their largest letter, as well as for 
$\I_n^{\Sigma}$ 
for certain collections $\Sigma$ of patterns.  For $\Sigma = \{ 2413, 3142 \}$, our result ensures that 
the complex of separable injective words over the alphabet $[n]$ 
is pure of dimension $n-1$.  These properties regarding purity will be needed for our shellability results later in the paper.

We will prove that for all patterns $\sigma $ of length at most 3, as well as most patterns of length 4, the complex $\I^\sigma_n$ is shellable. We will also characterize the facets closing off spheres during these shellings of $\I_n^{\sigma }$, called {\em homology facets}. Counting these facets allows us  to calculate the reduced Euler characteristic.  For the patterns $312$ and $321$, and more generally for any pattern $\sigma$ of any length  starting with its largest letter, we prove that the lexicographic order of the facets gives a shelling of $\I_n^{\sigma }$.  We use symmetries to deduce shellability of $\I_n^{\sigma }$ for all other patterns $\sigma $ of length 3, and for all but four of the patterns of length 4.  We then give a different proof of shellability that covers two of these four patterns, leaving the question open for just two patterns of length~4.  

Now back to our identity~\eqref{eq:Riordan-Catalan}.  By counting faces of each dimension in $\I_n^{\sigma }$, it is easy to see that, for any $\sigma\in\S_3$, the reduced Euler characteristic of $\I_n^{\sigma }$ is equal to
$\sum_{i=0}^n(-1)^{i-1} \binom{n}{i}C_{i}$.
Indeed, the number of permutations in $\S_i$ avoiding any given $\sigma\in\S_3$ is $C_i$, and hence that the number of $(i-1)$-dimensional faces in $\I_n^{\sigma }$ is $\binom{n}{i}C_{i}$.   From our result  that 
the complex $\I_n^{\sigma }$ 
is shellable for any $\sigma\in\S_3$, 
we deduce that 
the absolute value of the reduced Euler characteristic of $\I_n^{\sigma }$ equals the rank of the top homology of 
$\I_n^{\sigma }$. 
Moreover, it equals the number of homology facets in the shelling.  For $\sigma \in \{312, 321\}$, the lexicographic order is a shelling order, and we prove that the number of homology facets equals the number of  $\sigma$-avoiding permutations in $\S_n$ such that each  left-to-right maximum is immediately followed by a smaller entry. 
In Proposition~\ref{Riordan-as-Euler-char}, we deduce from  a bijection of Foata that such permutations are counted by the Riordan numbers, implying that the rank of the top homology (or equivalently the absolute value of the reduced Euler characteristic) for $\I_n^{\sigma }$ for any $\sigma\in\S_3$ is the $n$th Riordan number.
Multiplying the face number expression for the  reduced Euler characteristic by its sign of $(-1)^{n-1}$ then yields the identity~(\ref{eq:Riordan-Catalan}).

One of the highlights of this paper is the sort of interplay between topological and enumerative combinatorics illustrated (in one direction) in the above discussion of Equation~\eqref{eq:Riordan-Catalan}.   There is a rich history of using topological combinatorics (and especially shellability theory)  to deduce formulas for reduced Euler characteristic and thereby calculate M\"obius functions (see e.g.~\cite{St-EC1} for the technique, \cite{BW-Bruhat} for the example of Bruhat order, and \cite[Section 3.8]{St-EC1} for the example of the partition lattice, expressed in the language of R-labellings).  We are not aware of any past work going in the other direction in the manner that we do in parts of this paper --- using enumerative results to prove shellability. Our interest in utilizing and further developing this bridge between topology and combinatorics, and particularly in applying it  to  complexes of pattern avoiding  injective words,  stemmed in part from a hope that adding a  topological perspective to the mix  could lead to new enumerative results as well as interesting  new questions regarding  permutation patterns.    We list many of these  
questions in Section~\ref{sec:open}.

Going in one direction, we use topological results which we establish in this paper, specifically shellability of the complex of pattern-avoiding injective words for large families of individual  patterns (including e.g.\ $4213$, $4132$, $4321$, $4312$ and $4123$), to deduce refined enumeration formulas for permutations avoiding these patterns with respect to a permutation statistic, namely the number of left-to-right maxima that are not descents --- we call such entries {\em ephemeral records}.
Going in the other direction,  we introduce  a new technique that allows us to deduce  shellability of a Boolean cell  complex  from enumerative results regarding what we call {\em blockers} in the complex, for a choice of total order on its maximal faces. We apply this technique to 
the complex of separable injective words (i.e., those avoiding both $2413$ and $3142$) to deduce shellability in this case.  Specifically, we show that the lexicographic order of the facets is a shelling.

Further analysis of our shelling for $\I_n^{\sigma }$ for patterns $\sigma $ whose first or last   letter is largest or smallest allows us  to derive    an explicit description of the homology facets 
in this shelling, i.e.\ of a set of facets having the 
property that they index  the spheres in the wedge of spheres comprising $\I_n^{\sigma }$ in such a manner that their duals form a cohomology basis.  
Using our combinatorial description of these homology facets as our starting point, we construct an explicit homology basis for 
$\I_n^{\sigma }$ for each such $\sigma $.
This result, which appears as Theorems~\ref{homology-basis-theorem} and~\ref{homology-theorem-two}, is one of the main contributions of this paper, together with our various enumeration formulas and results regarding shellability.

The way we introduce and utilize the notion of blockers 
to prove shellability is seemingly new, though it is a relatively minor variation on ideas already in the literature; nonetheless, it is a surprisingly useful new twist on familiar ideas.  Given any total order $F_1,\dots ,F_m$ on the maximal faces of a pure Boolean cell complex, we first introduce a statistic on the facets: the {\em number of blockers} 
of $F_j$ is the number of codimension one faces of $F_j$ that are shared with earlier facets.  Next define the {\em blocker vector} 
$(B_0,B_1,\dots )$ of our facet ordering  by letting $B_i$ be the number of facets having $i$ blockers.  We prove in Lemma~\ref{lem:hvector_implies_shelling} that any pure Boolean cell complex admitting  a facet ordering  whose resulting blocker vector equals the $h$-vector of the complex is shellable, and that this facet ordering is a shelling order.  We use this method to prove that the complex of separable injective words is shellable.

The paper is structured as follows. In Section~\ref{bg-section} we provide background on permutations and topological combinatorics.   
In Section~\ref{sec:new_technique} we introduce a new technique to prove shellability of a complex via enumeration of its facets with respect to the blocker vector statistic.  
In Section~\ref{sec:In} we characterize blockers in the full complex of injective words, and use this to determine all of the $h$-numbers for $\I_n$. We also give a general result involving blockers that is used later in the paper when we later add pattern-avoidance restrictions.
In Section~\ref{sec:GFs} we provide an identity relating the $h$-polynomials of the complex $\I_n^\sigma$ with the generating function for $\sigma$-avoiding permutations.  In Section~\ref{sec:purity}, we  prove that $\I_n^{\sigma }$ is pure of dimension $n-1$ for two large families of patterns, and that this property also holds for $\I_n^{\Sigma}$ when $\Sigma = \{ 2413, 3142 \}$.

In Section~\ref{shell-subcomplex} we prove shellability of the complex $\I_n^{\sigma}$ for any pattern $\sigma$ that starts or ends with its smallest or its largest letter.
We show that when $\sigma$ starts with its largest letter, the lexicographic order of the facets of $\I_n^{\sigma}$ is a shelling order. This allows us to interpret the $h$-polynomial of this complex as enumerating $\sigma$-avoiding permutations with respect to the number of ephemeral records.
In Section~\ref{sec:length3}, we deduce the shellability of $\I_n^{\sigma}$ for any pattern $\sigma$ of length 3, we obtain a generating function for its $h$-polynomials, and show that the rank of their top homology is given by the Riordan numbers.
In Section~\ref{sec:length4}, we likewise handle most patterns of length 4, some of which require more subtle arguments. We also derive new results about the refined enumeration of permutations avoiding patterns of length 4.

 In Section~\ref{sec:separable} we use the technique from Section~\ref{sec:new_technique} to show that the complex of separable injective words is shellable.
In Section~\ref{homology-section}, we use our shellings and our characterizations of homology facets in these shellings  to construct  a homology basis for the complex $\I_n^{\sigma}$ when $\sigma$ starts or ends with its largest or smallest letter.
We conclude by proposing  some open questions in Section~\ref{sec:open}.

 \section{Background}\label{bg-section}

\subsection{Permutations, patterns, and enumeration}\label{sec:permutations}
We denote by $\S_n$ the set of permutations of $[n]$. 
\begin{defn}\label{def:records} For a permutation $\pi=\pi_1\pi_2\dots\pi_n\in \S_n$, we say that $\pi_j$ is a {\em record} (or a {\em left-to-right maximum}) if $\pi_i<\pi_j$ for all $i<j$.
We say that $\pi_j$ is a {\em descent} if it is immediately followed by a smaller entry, that is, $\pi_j>\pi_{j+1}$.
A record $\pi_j$ will be called an {\em ephemeral record} if it is not a descent (that is, either $j=n$ or $\pi_j<\pi_{j+1}$). 
Denote by $\er(\pi)$ the number of ephemeral records of $\pi$.
\end{defn}

Let $\L_n$ denote the set of permutations in $\S_n$ that have no ephemeral records, namely, those where every record is also a descent. For example, the records of $\pi=34261587$ are $3,4,6,8$. Among these, $3$ is the only ephemeral record, so $\er(34261587)=1$. In particular, $34261587\notin\L_8$, but $43261587\in\L_8$.

Denote by $\S_n(\sigma)$ the set of $\sigma$-avoiding permutations in $\S_n$, namely, those containing no subsequence that is order-isomorphic to $\sigma$.
Similarly, denote by $\L_n(\sigma)$ the set of $\sigma$-avoiding permutations in $\L_n$. 
Denote the ordinary generating function for $\sigma$-avoiding permutations by
\begin{equation}\label{eq:Psigma}
P_\sigma(x)=\sum_{n=0}^\infty |\S_n(\sigma)|\,x^n.
\end{equation}
We say that two patterns $\sigma,\tau\in\S_k$ are {\em Wilf-equivalent} if
$P_\sigma(x)=P_\tau(x)$.

It is well known~\cite{SS} that, for 
every $\sigma\in \S_3$, we have $|\S_n(\sigma)|=C_n=\frac{1}{n+1}\binom{2n}{n}$, the $n$th Catalan number, and so
$$P_\sigma(x)=\frac{1-\sqrt{1-4x}}{2x}.$$
We denote this generating function by $C(x)$. 
In particular, all the patterns of length 3 are Wilf-equivalent.

We will use $R_n$ to denote the $n$th Riordan number. The first few terms of this sequence, 
starting at $n=0$, are 
$$ 1,0,1,1,3,6, 15, 36, 91, 232,\dots .$$ 
The Riordan numbers, which appear as sequence A005043 in \cite{OEIS}, 
arise in the multiple enumeration problems involving restricted lattice paths and set partitions. Their generating function is 
\begin{equation}\label{eq:Riordan}
\sum_{n=0}^\infty R_n\,x^n=\frac{2}{1+x+\sqrt{(1+x)(1-3x)}}.
\end{equation}

\subsection{Boolean cell complexes and other related terminology}\label{Boolean-section}

 The notion  of {\em cell complex of Boolean type}, also known as {\em Boolean cell complex},  is a generalization of the concept  of simplicial complex.  Unlike in simplicial complexes, multiple faces are allowed to  have the same set of vertices. 
 Boolean cell complexes share many pleasant properties of simplicial complexes (see e.g.\ \cite{Hersh,Rei}).
  Throughout the remainder of  Section~\ref{bg-section}, we let $\Delta$ denote a cell complex of Boolean type, and we set $d=\dim(\Delta) + 1$.  
  
More formally,  a Boolean cell complex is any regular CW complex 
 whose face poset has a unique minimal element and also has  the property that each of its lower intervals is a Boolean lattice.   The face posets of Boolean cell complexes are known as {\em simplicial posets}. 
 The notion of  Boolean cell complex was introduced in \cite{Bj,GS} and the notion of simplicial poset was introduced in \cite{St}.

\begin{ex}
The complex $\I_n$ is a cell complex of Boolean type, and so is the subcomplexes $\I_n^{\sigma }$ for any pattern $\sigma$. Note that $\I_n$ is not a simplicial complex if $n\ge 2$, since both of the edges $12$ and $21$ have the same two vertices $1$ and $2$.    
On the other hand, the closure of any cell has the combinatorial type of a simplex, the hallmark feature of a cell complex of Boolean type.
\end{ex}

\begin{rk}
In this section we state some results from the literature for the general case of Boolean cell complexes, even though the quoted results were originally proven only for simplicial complexes.  In each of these cases, it is straightforward to check that  the results do hold in this greater generality by exactly the same reasoning as in simplicial complexes.  We note that  \cite{GS} has a similar disclaimer, as it deals with the same issue.
\end{rk}

\begin{defn}
For each cell $F$ in $\Delta$, denote by $\overline{F}$ the closure of $F$, namely the subcomplex consisting of $F$ and all the cells appearing below $F$ in the face poset.  
\end{defn}

\begin{defn}
We say that $\Delta$ is {\em pure} if all its maximal faces (called {\em facets}) have the same dimension as each other.  
\end{defn}

\begin{ex}
The complex $\I_n$ is pure of dimension $n-1$.  We also prove that $\I^\sigma_n$ is pure of dimension $n-1$ for many patterns $\sigma $, by showing that all its maximal faces are permutations in $\S_n$.  We are not aware of any patterns $\sigma $ for which $\I_n^{\sigma }$ is not pure of dimension $n-1$. 
\end{ex}

 \subsection{$f$-numbers and $h$-numbers}\label{sec:fh}
The $f$-number $f_i=f_i(\Delta)$ is the number of $i$-dimensional faces of $\Delta$.
 We make the convention of including the empty face and treating it as a $(-1)$-dimensional face, so that $f_{-1} = 1$.

 When $\Delta $ is pure, one also associates to $\Delta $ a list of numbers called the $h$-numbers of $\Delta $.
 The $h$-number $h_i=h_i(\Delta)$ is defined  by the formula 
$$h_k = \sum_{i=0}^k(-1)^{k-i}\binom{d-i}{k-i} f_{i-1}$$ for $0\le k\le d$.   Conversely, the $f$-numbers of a pure complex   may be determined from the $h$-numbers via the formula
\begin{equation}\label{eq:f-from-h}
f_{j-1} = \sum_{i=0}^j{d-i\choose d-j}h_i
\end{equation}
for $0\le j\le d$.

The {\em $h$-polynomial} of $\Delta$ is defined as $h(\Delta,t)=\sum_{i=0}^d h_i t^i$. It will be convenient to also consider the {\em reverse $h$-polynomial} 
\begin{equation}\label{eq:h-reverse}
\bh(\Delta,t)=\sum_{i=0}^d h_i t^{d-i}=t^d h(\Delta,1/t).
\end{equation}
Equation~\eqref{eq:f-from-h} can be written in terms of polynomials as
\begin{equation}\label{eq:f-h}
\sum_{i=0}^d  f_{i-1} (t-1)^{d-i}  = \bh(\Delta,t).
\end{equation}

The $f$-vector (resp.\ $h$-vector) of of a $(d-1)$-dimensional Boolean cell complex $\Delta$ is the vector 
$(f_{-1}(\Delta ),f_0(\Delta ),\dots ,f_{d-1}(\Delta ))$ (resp.\ 
$(h_0(\Delta ),\dots ,h_d(\Delta )))$.

\begin{ex}
The simplicial complex on five vertices $\{ v_1,v_2,v_3,v_4,v_5\} $ whose facets are $\{ v_1,v_2,v_3\} $ and $\{ v_3,v_4,v_5 \} $ has $h$-vector $(1,2,-1)$. This illustrates that $h$-numbers may be negative for pure complexes which are not shellable.
\end{ex}

For more background on $f$-numbers and $h$-numbers, we refer the reader to~\cite{St2}.

 \subsection{Reduced Euler characteristic and Betti numbers}\label{euler-background}

 The {\em reduced Euler characteristic} of $\Delta$ is 
 \begin{equation}\label{eq:reduced}
 \tilde{\chi }(\Delta )=-1 + \sum_{i=0}^{d-1}(-1)^i f_i(\Delta ).
 \end{equation}
 Compared to the usual Euler characteristic, the additional $-1$ term arises from regarding the empty set as a $(-1)$-dimensional face.  
 
Recall that the {\em $i$th Betti number}  $\beta_i (\Delta )$ is the rank of the $i$th homology group of $\Delta $ with integer coefficients.  
 It is well known that
 $\tilde{\chi }(\Delta ) =-1 + \sum_{i=0}^{d-1} (-1)^i \beta_i(\Delta )$.

If $\Delta$ is pure, evaluating Equation~\eqref{eq:f-h} at $t=0$ yields 
\begin{equation}\label{eq:chi-h}
\tilde{\chi}(\Delta) = (-1)^{d-1}h_d(\Delta).
\end{equation}

\subsection{Shellability and some of its consequences}\label{sec:shellability}

A {\em shelling} of a 
Boolean cell complex $\Delta $ is a total order $F_1,\dots ,F_m$ 
on the maximal faces of $\Delta $  such that, for each $j\ge 2$, the maximal faces in $\overline{F}_j \cap (\cup_{i<j} \overline{F}_i)$ are all codimension one faces of $F_j$.    A complex is {\em shellable} if it admits a shelling.   Shellability of $\Delta$ implies that it is homotopy equivalent to a wedge of spheres.  When $ \Delta $ is pure and shellable, then this is a wedge of top-dimensional spheres, implying 
all of the  Betti numbers are $0$ except $\beta_0$ which equals 1 and the top Betti number which (up to a sign) equals the reduced Euler characteristic of the complex.

\begin{rk}\label{rk:combin-shelling}
A shelling may equivalently be defined more combinatorially  as a total ordering $F_1,\dots ,F_m$ on the facets of $\Delta $ such that, for each $j\ge 2$, there is a unique containment-wise minimal face $G_j$ in the set $\overline{F}_j\setminus (\cup_{i<j}\overline{F}_i)$ of faces of $F_j$.   This face $G_j$ is called the {\em restriction face} of $F_j$.   

One may verify the equivalence of the two definitions  by noting that the restriction face  $G_j$ for the shelling step $F_j$ consists of exactly those vertices $v$ in $F_j$ such that $F_j\setminus \{ v \} $
is a codimension one face of $F_j $ that is present in $\overline{F_j}\cap (\cup_{i<j} \overline{F_j})$.

 Using the more combinatorial  definition of shelling, it is well known 
 for pure complexes 
 that  $h_i$ is the number of shelling steps $F_j$ for which $G_j$ has exactly $i$ vertices in it.  See Remark~\ref{shelling-h-numbers} for one way of justifying this fact that we will utilize later.  
\end{rk}

 \begin{defn}
 A pure Boolean cell complex $\Delta $ is {\em partitionable} if its face poset may be decomposed into a disjoint union of Boolean lattices 
 whose maximal elements are the facets of $\Delta $.  A {\em partitioning} is any such decomposition. 
 \end{defn}

 \begin{rk}\label{rk:partitionable}
It is well known (see e.g.\ ~\cite[Chapter III, Section 2]{St2}) 
that any  pure 
shellable
Boolean cell complex $\Delta $ is partitionable.  Given a shelling $F_1,\dots ,F_m$,
the intervals $[G_j,F_j]$ in the face poset of $\Delta $ indexed by the facets $F_j\in \Delta $ and the corresponding restriction faces $G_j$ 
comprise a partitioning of $\Delta $. 
 \end{rk}

 \begin{defn}
A {\em homology facet} in a shelling is any facet $F_j$ such that
 $\overline{F}_j \cap (\cup_{i<j} \overline{F}_i)$ equals the  set of all faces in the 
 boundary of $\overline{F}_j$.  
 In other words, $F_j$  is any facet whose codimension one faces are all shared with earlier facets. 
\end{defn}

It is well known that any shellable complex is homotopy equivalent to a wedge of spheres in which the number of $i$-dimensional spheres equals the number of $i$-dimensional homology facets in any shelling of the complex. 
Thus, for any pure, shellable Boolean cell complex $\Delta $,
one has 
 $$\tilde{\chi }(\Delta) = (-1)^{d-1}\beta_{d-1}(\Delta ).$$  In fact, one may calculate $|\tilde{\chi }(\Delta )|$ in this case simply by counting the homology facets of the shelling, since these index the spheres in the wedge of spheres comprising $\Delta $.   This is a technique which has been used many times to calculate reduced Euler characteristic and is one we employ as well.  

Wachs showed in~\cite{Wachs-cohomology} that  any (pure or nonpure)  shellable complex $\Delta $ has a cohomology basis such that the homology facets of dimension $i$ naturally give rise to a basis  for the $i$th cohomology group of $\Delta $; specifically, she took the duals of the homology facets as the elements of a cohomology basis.  

It is also worth noting for pure shellable complexes that   $h_d(\Delta)=\beta_{d-1}(\Delta)$, since both quantities calculate  the rank of the top homology group of $\Delta$  by virtue of both  counting the homology facets in the shelling.

We next introduce a notion that will play a prominent role in our proofs of shellability as well as in our enumerative results  calculating $h$-numbers.  

\begin{defn}\label{def:blockers-facet}
Given a facet ordering $F_1,\dots ,F_m$ for a pure cell complex of Boolean type, define the {\em blockers} of a facet $F_j$ with respect to this facet ordering to be those vertices $v\in F_j$ such that the codimension one facet $F_j\setminus \{ v\} $ in $\overline{F}_j$ is also in the closure of some earlier facet $F_i$, implying $F_j\setminus \{ v\} $ is a face in $\overline{F_j} \cap (\cup_{i<j} \overline{F}_i$).
\end{defn}

\begin{rk}
Whenever $F_1,\dots ,F_m$ is a shelling, notice  that the restriction face $G_j$ of any  facet $F_j$ consists exactly of the blockers for $F_j$.
\end{rk}

The term ``blocker'' is new, though the idea of blockers appears implicitly in many prior works on shellability, for instance with the descents in an EL-labeling or CL-labeling giving rise to the blockers there.   
We will make use of the fact that blockers are well defined  even for facet orderings that are not shellings.  For instance, the  nonshellable simplicial complex on 5 vertices whose facets are $F_1 = \{ v_1,v_2,v_3\} $ and $F_2 = \{ v_3,v_4,v_5\} $ has no blockers in either facet in the ordering $F_1,F_2$ on the facets.   

Next is a well-known fact that  we rephrase using the language of blockers.  We will make extensive use of this reformulation of this fact.

\begin{lem}\label{h-via-blockers}
If a pure Boolean cell complex $\Delta $ is shellable, then $h_i(\Delta )$ is a nonnegative integer which equals the number of shelling steps  given by facets 
having exactly $i$ blockers.
\end{lem}

\begin{proof}
Using the combinatorial definition of shellability, 
it is well known that $h_i(\Delta )$ equals the number of shelling steps $F_j$ whose restriction faces  $G_j$ have exactly $i$ vertices in them.  But each restriction face $G_j$ consists of exactly the blockers in $F_j$, as explained above, so the result follows.
\end{proof}

\begin{rk}
For the complex $\I_n^{\sigma }$ of injective words avoiding a pattern $\sigma$ starting with its largest letter, we will prove that $\I_n^{\sigma }$  has a shelling in which its blockers are exactly the entries that are not ephemeral records (see Definition~\ref{def:records}).
\end{rk}

\subsection{Shellability of the complex of injective words}

 The complex of injective words $\I_n$ over the alphabet $[n]$ is  a pure, shellable  $(n-1)$-dimensional cell complex of Boolean type.  
 The original proof of shellability of this complex, by Bj\"orner and Wachs in \cite[Theorem 6.1, Part 1]{BW}, was based on a different  method to what we will use.  In \cite{BW} they gave a recursive coatom ordering for the face poset of $\I_n$,  a technique which they proved implies dual CL-shellability of the poset (which by another result of theirs this implies shellability of the complex itself).  Their coatom ordering gives rise to the same shelling order that we use later, namely it orders  the permutations in $\S_n$, written in one-line notation,  lexicographically.

\section{A technique to prove shellability via blocker enumeration}
\label{sec:new_technique}

This section gives a new technique to use enumeration to prove shellability.  We will apply this method   in Section~\ref{sec:separable} to the complex of separable injective words, enabling us to show that this complex is shellable.  First we describe a way of thinking about $h$-numbers that will be essential to our enumerative approach to shellability.

The face poset  of any pure shellable Boolean cell complex  may be  decomposed into nonoverlapping Boolean lattices whose maximal elements are all facets of the complex by assigning each face to the earliest facet having the face in its closure.  This is the partitioning discussed in Remark~\ref{rk:partitionable} that is given by the shelling.  
The $h$-number $h_i$ counts how many of these Boolean lattices have their minimal element at rank $i$.  Note that the empty face, which is $(-1)$-dimensional, is the only element of rank 0 in the face poset, ensuring for any partitionable complex 
that $h_0 = 1$. 
Using this partitioning  to calculate $f$-numbers from $h$-numbers  yields the useful observation described in the  next remark.

\begin{rk}\label{shelling-h-numbers}
When a Boolean cell complex $\Delta $ is pure of dimension $d-1$ and is  shellable, Equation~\eqref{eq:f-from-h} has the following combinatorial explanation.  
When the facet  $F_j$ in a shelling $F_1,\dots ,F_m$ has as
its blockers $\{ v_{i_1},\dots ,v_{i_s}\}$, then we already made the following observations  in Remark~\ref{rk:combin-shelling} that we will now utilize:
\begin{enumerate}
    \item 
    The 
set of faces in $\overline{F}_j \setminus (\cup_{i<j}\overline{F}_i)$ is a Boolean lattice  whose top element is  $F_j$ and whose   minimal element is  $\{ v_{i_1},\dots ,v_{i_s}\} $.
\item
The 
shelling step given by  $F_j$ has restriction face 
$\{ v_{i_1},\dots ,v_{i_s}\} $ and hence 
contributes 1 to $h_s$.
\end{enumerate}
If we assign each face $\gamma  \in \Delta $ 
to the earliest facet $F_j$ with $\gamma  \subseteq \overline{F}_j$, then 
$F_j$ contributes to $\Delta $ exactly the  set of faces in $$\overline{F}_j\setminus (\cup_{i<j}\overline{F}_i),$$
namely all of the subsets of the set of $d$ vertices in $F_j$ that include all $s$ blockers 
of  $F_j$.  Thus, the contribution of $1 $ to $h_s$ from $F_j$ translates to  $F_j$ contributing exactly ${d-s\choose k-s}$ faces having $k$ vertices in them for each $k\ge s$.  From this analysis,  
we deduce that
$$f_{k-1}=\sum_{i=0}^k{d-i\choose k-i}h_i$$ by summing up the contributions to the number of $(k-1)$-faces from the various facets.  This confirms Equation~\eqref{eq:f-from-h} in a combinatorial manner for pure shellable complexes.   
\end{rk}

\begin{defn}
Given a total order $F_1,\dots ,F_m$ on the facets of a pure $(d-1)$-dimensional Boolean cell complex $\Delta $, define the {\em blocker vector}, denoted $(B_0,B_1,\dots ,B_{d} )$, 
for this facet ordering to be the vector in which $B_i$ counts the number of facets having exactly $i$ blockers.  
\end{defn}

\begin{lem}\label{lem:hvector_implies_shelling}
Let $F_1,F_2,\dots ,F_m$ be a total order 
on the facets of a pure $(d-1)$-dimensional Boolean cell complex 
$\Delta $. 
Let $(h_0,h_1,\dots ,h_d)$ be the $h$-vector for $\Delta $, and let 
$(B_0,B_1,\dots ,B_d)$ denote the blocker vector given by the facet ordering $F_1,\dots ,F_m$.  
If $h_i=B_i$ for all $i$, then $F_1,\dots ,F_m$ 
is a shelling order for $\Delta $.
\end{lem}

\begin{proof}
Consider  any $F_j$ and its blocker set $\mathcal{B} = \{ v_{i_1},\dots ,v_{i_s}\} $.
Then every face in 
$\overline{F}_j \setminus  (\cup_{i<j} \overline{F}_i)$ must include all of the vertices in $\mathcal{B}$, since every codimension one face of $F_j$  obtained by leaving out any one of these blockers is shared with an earlier facet, implying all faces of $F_j$ leaving out one or more of these blockers are also shared with earlier facets.  This yields
$f_{k-1}\le \sum_{i=0}^k{d-i\choose k-i}B_i$ for $0\le k\le d$, by reasoning similar to the analysis  in Remark~\ref{shelling-h-numbers}, namely by noting that each facet $F_j$ which contributes $1$ to $B_i$  contributes at most ${d-i \choose k-i}$ to $f_{k-1}$.  Moreover, equality is achieved if and only if each face in $\overline{F_j}$ containing all the blockers of $F_j$ is in $\overline{F}_j\setminus (\cup_{i<j}\overline{F}_i)$ for $j=1,2,\dots ,m$. 
But our $B_i=h_i$ assumption implies that the right hand side of this inequality equals $\sum_{i=0}^k{d-i\choose k-i}h_i$, which by Equation~(\ref{eq:f-from-h}) is exactly $f_{k-1}$.   Thus, our $B_i=h_i$ assumption for all $i$  implies this inequality is an equality for each $j$.  But this ensures for each $j$ that 
$\overline{F}_j\setminus (\cup_{i<j}\overline{F}_i)$ consists of exactly the faces contained in the face poset interval
$[G_j,F_j]$ where $G_j$ is the face consisting of the blockers of $F_j$.  In other words, for each $j$ we deduce that  $G_j$ is the unique minimal face in $\overline{F}_j\setminus (\cup_{i<j}\overline{F}_i)$, implying $F_1,\dots ,F_m$ is a shelling
for  $\Delta $ using the combinatorial definition of shellability.
\end{proof}

This can also be strengthened in the manner described next.  The point of this strengthening is  that it  could allow one to choose a facet ordering of a pure Boolean cell complex, identify  some  blockers with respect to this  ordering, determine the $h$ numbers of the complex,  and then  deduce from the next result  that there are no additional blockers and that the chosen facet ordering is a shelling order.

In what follows, we will use the rather informal term  \emph{known blocker} of a facet $F_j$ with respect to a facet ordering $F_1,\dots ,F_m$ of a pure simplicial complex to refer to  any  vertex $v$ of $F_j$ that is known to have the property that it is a blocker.
Note that not every blocker must be a known blocker.  The point is that one might identify any subset of the set of  blockers and call them the  known blockers, then one might attempt to use the next result to prove that  these are all the blockers.

\begin{lem}\label{lem:more-general-version}
Let $F_1,\dots ,F_m$ be a total order on the facets of a pure $(d-1)$-dimensional Boolean cell complex $\Delta $. 
Let $(h_0,h_1,\dots ,h_d)$ be the $h$-vector of $\Delta $, and let $(K_0,\dots ,K_d)$ denote  the vector in which $K_i$ counts  the number of facets  with exactly $i$ known blockers.  
If $h_i=K_i$ for all $i$, then $F_1,\dots ,F_m$ is a shelling order in which the  known blockers are all of the blockers.
\end{lem}

\begin{proof}
If $F_j$ has known blockers $\{ v_{i_1},\dots ,v_{i_r}\} $, then  
the only faces in $\overline{F}_j$ that could possibly be contained in $\overline{F}_j \setminus (\cup_{i<j}\overline{F}_i)$ are faces of $F_j$ containing all these known blockers.  But we need all the  faces in $\overline{F}_j$ which contain all its known blockers  to be in 
$\overline{F}_j\setminus  (\cup_{i<j}\overline{F}_i)$ in order for $\Delta $ to have enough faces for its face numbers to  yield the $h$-vector $(h_0,\dots ,h_d)$, by the same reasoning as in the proof of Lemma~\ref{lem:hvector_implies_shelling}. 

This implies that there are no other blockers besides the known blockers, since leaving out any one vertex  $v$ from $F_j$ that is not in the set 
$\{ v_{i_1},\dots ,v_{i_r}\} $ of known blockers of $F_j$  must yield a face $F_j\setminus \{ v\} $ that contains all these known blockers and hence must not be shared with any earlier facet.   This shows for  $v\not\in \{ v_{i_1},\dots ,v_{i_r}\} $ that $v$  cannot be a blocker.    

This also implies that $F_1,\dots ,F_m$ is a shelling order since we cannot leave out any set of two or more vertices from $F_j$ and get a face $G$ shared with an earlier facet $F_i$ with $i<j$ unless $G$ is missing one or more of our known blockers, since otherwise $F$ would  include all our known blockers and not be in $\overline{F}_j\setminus (\cup_{i<j}\overline{F}_i)$, yielding a contradiction.  This shows for each $j>1$ that the insertion of $F_j$ must be a shelling step with exactly the known blockers as its blockers.
\end{proof}

\section{An $h$-number formula for the complex of injective words and shellability for many of its pattern-avoiding subcomplexes}
\label{sec:In}

In this section, we characterize the blockers in
the shelling 
for $\I_n$ from \cite{BW}, in which permutations (written in one line notation) are ordered lexicographically.
Then we show how to use this characterization to prove shellability for any subcomplex $\I_n^{\sigma }$ 
 whose blockers are sufficiently similar to those in the full complex of injective words $\I_n$.  We will use this 
 result later in the paper to prove shellability of $\I_n^{\sigma } $ for large classes of patterns $\sigma $.     
Using our  description of the  blockers in $\I_n$  in conjunction with a bijection of Foata, we also  deduce  a seemingly new 
combinatorial formula for each $h$-number of  $\I_n$.

We denote by $\lex$ the lexicographic order on permutations.
In Definition~\ref{def:blockers-facet} we defined blockers in the facets of a shelling of a cell complex of Boolean type. 
For the complex of injective words $\I_n$, given a fixed ordering of its facets, the definition of a blocker can be restated as follows. A blocker for this complex will be called an {\em unrestricted blocker}, to emphasize that the words in the complex have no pattern-avoidance restrictions.

\begin{defn}\label{def:unrestricted_blocker} 
Fix a total order $\prec$ (which we will assume to be $\lex$ unless otherwise stated) 
on $\S_n$. Given $\pi\in\S_n$, an entry $\pi_j$ is an {\em unrestricted blocker} (with respect to $\prec$) if $\pi_1\dots \widehat{\pi_j}\dots \pi_n$ is a subsequence of some permutation $\pi'\in\S_n$ such that $\pi'\prec\pi$; equivalently, if 
$\pi_j$ can be moved within $\pi_1\pi_2\dots\pi_n$ (while keeping the other entries in the same order) to create a smaller permutation.
\end{defn}

Next we characterize unrestricted blockers with respect to the lexicographic order.  

\begin{lem}\label{lem:unrestricted_blockers}
   Let $\pi\in\S_n$. The entry $\pi_j$ is an unrestricted blocker with respect to $\lex$ if and only if either 
    \begin{itemize}
    \item[(a)] $\pi_j$ has a larger entry somewhere to its left (i.e., $\pi_i>\pi_j$ for some $i<j$), or
    \item[(b)] $\pi_i$ has a smaller entry immediately to its right (i.e., $j<n$ and $\pi_j>\pi_{j+1}$).
    \end{itemize}
\end{lem}

\begin{proof}
By Definition~\ref{def:unrestricted_blocker}, $\pi_j$ is an unrestricted blocker with respect to $\lex$ if it can be moved within $\pi$ to create a lexicographically smaller permutation.

The positions where $\pi_j$ can be moved to in order to create a lexicographically smaller permutation are as follows:
\begin{itemize}
    \item it can be placed immediately before any larger entry to the left of it, and
    \item if $\pi_j>\pi_{j+1}$, it can be moved anywhere to the right of its current position.
\end{itemize}
The first kind of position exists if and only if $\pi_j$ has a larger entry somewhere to its left (condition (a)). The second kind of position exists if and only if $j<n$ and $\pi_j>\pi_{j+1}$ (condition (b)).
\end{proof}

With the terminology from Section~\ref{sec:permutations}, Lemma~\ref{lem:unrestricted_blockers} states that the unrestricted blockers of $\pi\in\S_n$ with respect to $\lex$ are precisely those entries that are not ephemeral records. 

\begin{rk}
Interestingly, the entries $\pi_j$ in a permutation $\pi = \pi_1\cdots \pi_n$ which are not ephemeral records in $\pi $ also play an important role in the proof in \cite{BW} that the face poset for $\I_n$ admits a recursive coatom ordering, namely their result which implies shellability of $\I_n$. 
\end{rk}

Next we give two seemingly new combinatorial interpretations of the $h$-numbers of the complex of injective words. 
For $\pi\in\S_n$, let $\fp(\pi)$ denote the number of fixed points of $\pi$, and recall that $\er(\pi)$ denotes the number of ephemeral records of $\pi$. Recall also that $d_n$ denotes the number of derangements in $\S_n$.

\begin{prop}\label{prop:h-numbers}
    For $0\le i\le n$, we have
    \[
        h_i(\I_n)=|\{\pi\in\S_n:\er(\pi)=n-i\}|
        =|\{\pi\in\S_n:\fp(\pi)=n-i\}|=\binom{n}{i}d_i.
    \]
\end{prop}

\begin{proof}
 The first equality follows from Lemmas~\ref{h-via-blockers} and~\ref{lem:unrestricted_blockers}, using the fact that the complex of injective words is shellable with respect to $\lex$.

    We will prove the second equality bijectively using Foata's first fundamental transformation \cite[Section 10.2]{Lothaire}.
    For any permutation $\pi\in\S_n$, its image $\hat\pi$ under this transformation is defined as follows. First, write $\pi$ in cycle form where each cycle starts with its largest element, and the cycles are listed from left to right by increasing initial elements. Then, remove the parentheses to obtain the one-line notation of $\hat\pi$.

    Note that the initial elements of the cycles of $\pi$ become the records of $\hat\pi$. Thus, fixed points (i.e., cycles of length one) of $\pi$ become records of $\hat\pi$ that are not followed a smaller entry, i.e., ephemeral records. This proves the second inequality. 

    The third equality is straightforward.
\end{proof}

When $i=n$, Proposition~\ref{prop:h-numbers} states that  $h_n(\I_n)=d_n$, a fact that was already observed in \cite{WebbR}.
In this case, Foata's transformation gives a bijection between the set of derangements in $\S_n$ and the set $\L_n$ of permutations with no ephemeral records.

Theorem~\ref{thm:shelling-via-unrestricted-blockers} below  recovers the shelling of \cite{BW}, but using a different proof which allows the result to be generalized to prove shellability for numerous subcomplexes of $\I_n$ as well.   
The definition of blockers for the complex $\I_n^\sigma$ can be restated as follows.

\begin{defn}\label{def:blocker} 
Let $\sigma\in\S_k$, and fix a total order $\prec$ 
on $\S_n(\sigma)$. Given $\pi\in\S_n(\sigma)$, an entry $\pi_j$ is a {\em $\sigma$-blocker} 
(with respect to $\prec$) if $\pi_1\dots \widehat{\pi_j}\dots \pi_n$ is a subsequence of some permutation $\pi'\in\S_n(\sigma)$ such that $\pi'\prec\pi$; equivalently, if 
$\pi_j$ can be moved within $\pi_1\dots\pi_n$ (while keeping the other entries in the same order) to create a smaller $\sigma$-avoiding permutation.
\end{defn}

The following simple observation will be useful later.

\begin{lem}\label{lem:no_gained_blockers}
Let $\sigma\in\S_k$, let $\prec$ be a total ordering on $\S_n(\sigma)$, and let $\pi\in\S_n(\sigma)$. If an entry $\pi_j$ is a $\sigma$-blocker (with respect to $\prec$), then it is also an unrestricted blocker (with respect to $\prec$).
\end{lem}

\begin{proof}
    By Definition~\ref{def:blocker}, $\pi_j$ is a $\sigma$-blocker if $\pi_1\dots \widehat{\pi_j}\dots \pi_n$ is a subsequence of some permutation $\pi'\in\S_n(\sigma)$ such that $\pi'\prec\pi$. But then, since $\pi'\in\S_n$ and $\pi'\prec\pi$, the entry $\pi_j$ is an unrestricted blocker as well.
\end{proof}

\begin{thm}\label{thm:shelling-via-unrestricted-blockers}
Consider any pattern $\sigma $ with the following two properties:
\begin{enumerate}
\item  $\I_n^{\sigma }$ is pure of dimension $n-1$.
\item The $\sigma $-blockers with respect to $\lex$ in any $\pi\in\S_n(\sigma)$ are exactly the unrestricted blockers of $\pi$ with respect to $\lex$.
\end{enumerate}
Then  the lexicographic ordering on the $\sigma $-avoiding permutations in $\S_n$ written in one-line notation is a shelling order for $\I_n^{\sigma }$.

By the same reasoning, the lexicographic ordering on all permutations in $\S_n$ written in one-line notation is a shelling order for $\I_n$.
\end{thm}

\begin{proof}
Our first requirement on $\sigma $ guarantees that the facets in $\I_n^{\sigma }$ are exactly the $\sigma $-avoiding permutations in $\S_n$.
By Lemma~\ref{lem:unrestricted_blockers}, the unrestricted blockers in any $\sigma$-avoiding permutation $\pi  = \pi_1\dots \pi_n$ with respect to $\lex$ are those letters $\pi_i$ which either have a larger letter to their left or have a smaller letter to their immediate right; in other words, entries that are not ephemeral records.   Our second requirement on $\sigma $ ensures that the unrestricted blockers are exactly the $\sigma $-blockers.  Henceforth in this proof, we simply call these the blockers of the $\sigma $-avoiding permutations.   

Now let $\pi\in\S_n(\sigma)$, and suppose that the face $\pi_1 \dots \hat{\pi}_{i_1}\dots \hat{\pi}_{i_k}\dots \pi_n$ is shared with a lexicographically earlier $\sigma$-avoiding permutation $\pi'$.    Suppose that none of the letters in the set
 $S = \{ \pi_{i_1},\dots ,\pi_{i_k }\} $ of omitted letters are blockers in $\pi$.    This implies that each of these letters is an ephemeral record, so in particular these letters comprise an increasing subsequence within $\pi$.  The leftmost letter among $\pi_{i_1},\dots ,\pi_{i_k} $ that shifts to a new position in $\pi'$ must shift rightward past a letter smaller than it  in order for $\pi'$ to be lexicographically smaller than $\pi$.  
Let $\pi_i $ be the leftmost letter where $\pi$ differs from $\pi'$, implying  $\pi_i ' < \pi_i $.   
Since $\pi_1\dots \pi_{i-1} = \pi_1 '\dots \pi_{i-1}'$, we must have some $j>i$ where $\pi_j = \pi_i'$ which means the letters
$\pi_i,\pi_{i+1},\dots ,\pi_{j-1}$ are all shifted rightward past  $\pi_j$ to obtain $\pi'$.   In particular $\{ \pi_i,\dots ,\pi_{j-1} \}  \subseteq S $
since $\pi_j \not \in S$ by virtue of $\pi_j$ not being a left-to-right maximum.
But $\pi_j < \pi_i \le \pi_{j-1}$  where $\pi_{j-1}\in S$, which contradicts 
$\pi_{j-1}$ not being a blocker in $\pi$.   The upshot of all this is that for our facet ordering $F_1,\dots ,F_m$ (which is the list of elements of $\S_n(\sigma)$ ordered lexicographically), each of the maximal faces in $\overline{F}_r \cap (\cup_{i<r}\overline{F}_i)$ for each $r$ is obtained by deleting a single blocker from $F_r$. Thus, each such maximal face is a codimension one face of $F_r$, completing the proof  that the lexicographic order on $\sigma $-avoiding permutations is  a shelling order.   

Also note that $\I_n$ itself vacuously meets these two  requirements we impose on $\sigma $, allowing the proof above  also to be used to deduce shellability of $\I_n $ itself with respect to the lexicographic order on permutations in $\S_n$  written in one-line notation.
\end{proof}

In Theorem~\ref{thm:shellability_start_with_k}, we will apply Theorem~\ref{thm:shelling-via-unrestricted-blockers} to a large class of permutation patterns.

\section{Generating functions for the $h$-numbers of $\I_n^\sigma$}
\label{sec:GFs}

 Throughout this section, we restrict attention to  complexes $\I_n^{\sigma }$ which 
are pure.  Our proofs here do not use this hypothesis, but that is the level of generality in which $h$-vectors are typically defined. For nonpure complexes, one could instead consider  the $h$-triangles introduced in~\cite{BW-nonpure}.

Let $f_i(\I_n^\sigma)$ and $h_i(\I_n^\sigma)$ denote the $f$-numbers and $h$-numbers, respectively, of the complex of $\sigma$-avoiding injective words $\I_n^\sigma$, and recall the definition of the reverse $h$-polynomial $\bh(\I_n^\sigma,t)$ from Equation~\eqref{eq:h-reverse}.
The following result expresses the generating function for these polynomials in terms of the generating function $P_\sigma(x)$ for $\sigma $-avoiding permutations  defined in Equation~\eqref{eq:Psigma}.

\begin{prop}\label{prop:GFh} For any pattern $\sigma$, we have
\[
    \sum_{n=0}^\infty \bh(\I_n^\sigma,t)\,x^n
=\frac{1}{1-(t-1)x}\,P_\sigma\left(\frac{x}{1-(t-1)x}\right),
\]
or equivalently,
\[
    \sum_{n=0}^\infty h(\I_n^\sigma,t)\,x^n
=\frac{1}{1-(1-t)x}\,P_\sigma\left(\frac{tx}{1-(1-t)x}\right);
\]
In particular, the generating function for the top $h$-number of the complexes $\I_n^\sigma$ is
\begin{equation}\label{eq:Euler_characteristic}
    \sum_{n=0}^\infty h_n(\I_n^\sigma)\,x^n
=\frac{1}{1+x}\,P_\sigma\left(\frac{x}{1+x}\right).
\end{equation}
\end{prop}

\begin{proof}
We have $$f_{i-1}(\I_n^\sigma)=\binom{n}{i}|\S_i(\sigma)|,$$
since faces of dimension $i-1$ are obtained by choosing $i$ distinct letters from $[n]$, and then listing them in any order that avoids~$\sigma$.
Thus, by Equation~\eqref{eq:f-h}, we can write 
\begin{equation}\label{eq:h-Si}
\bh(\I_n^\sigma,t)=\sum_{i=0}^n  \binom{n}{i}|\S_i(\sigma)| (t-1)^{n-i}.
\end{equation}

We deduce that the generating function for these reverse $h$-polynomials, after some manipulations, can be expressed as
\begin{align}
    \sum_{n=0}^\infty \bh(\I_n^\sigma,t)\,x^n
&=\sum_{n=0}^\infty \sum_{i=0}^n  \binom{n}{i}|\S_i(\sigma)| (t-1)^{n-i}\,x^n \nonumber\\
&=\sum_{i=0}^\infty |\S_i(\sigma)| \sum_{n=i}^\infty  \binom{n}{i} (t-1)^{n-i}\,x^n \nonumber\\
&=\sum_{i=0}^\infty |\S_i(\sigma)| \frac{x^i}{(1-(t-1)x)^{i+1}} \nonumber\\
&=\frac{1}{1-(t-1)x}\,P_\sigma\left(\frac{x}{1-(t-1)x}\right),
\label{eq:bh}
\end{align}
where in the third line we used the generating function identity $\frac{x^i}{1-x^{i+1}}=\sum_{n=i}^\infty \binom{n}{i}x^n$.

The generating function for $h(\I_n^\sigma,t)$ is obtained from Equation~\eqref{eq:bh} by making the substitution $(t,x)\mapsto(1/t,tx)$.
On the other hand, setting $t=0$ in this equation we obtain the generating function for $h_n(\I_n^\sigma)$.
\end{proof}

\section{Purity of $\I_n^{\sigma }$ for two large classes of patterns}\label{sec:purity}

In this section, we prove for 
two large classes of patterns $\sigma $ that $\I_n^{\sigma }$ is pure of dimension $n-1$; equivalently, that all of the maximal faces in $\I_n^{\sigma }$ are permutations in $\S_n$.

\subsection{Extremal-endpoint patterns} \label{sec:extremal-endpoint}
The first class we consider are the patterns $\sigma \in \S_k$ that start or end with their largest or smallest letter, that is, they satisfy $\{ \sigma_1,\sigma_k\} \cap \{ 1,k\} \neq \emptyset $.  We call such patterns {\em extremal-endpoint patterns}. Recall that we assume $k\ge2$ for all our patterns.  
Proving purity of dimension $n-1$ will be an important  ingredient to proving shellability of $\I_n^{\sigma }$ for such patterns $\sigma $ later in the paper.  
Let us start by proving purity for a subset of these patterns.

\begin{prop}\label{leftmost-pure}
Let $\sigma\in\S_k$  with $\sigma_1=k$, and suppose $k\ge2$.  Then every maximal face in $\I_n^{\sigma }$ is a permutation in $\S_n$, hence $\I_n^{\sigma }$ is pure of dimension $n-1$.
\end{prop}

\begin{proof} 
Let $\pi=\pi_1\dots\pi_r$ be a $\sigma$-avoiding injective word  of length $r<n$ over the alphabet $[n]$. To prove that $\pi$ cannot be a maximal face, we will show that it is possible to insert some letter $j\in[n]\setminus\{\pi_1,\dots,\pi_r\}$ into $\pi$ to create another $\sigma$-avoiding injective word $\pi'$.
We will show that, in fact, any such $j$ can be inserted in this way.

If $j$ is larger than all the letters in $\pi$, then  $\pi'=\pi_1\dots\pi_rj$ avoids $\sigma$. Indeed, since $\pi$ avoids $\sigma$, any occurrence of $\sigma$ in $\pi'$ would have to use the letter $j$, but this is not possible because $\sigma$ does not end with its largest letter (since $\sigma_1=k>\sigma_k$).

If some letters in $\pi$ are larger than $j$, let $\pi_i$ be the leftmost such letter. Then we claim that $\pi'=\pi_1\dots \pi_{i-1} j \pi_i\dots \pi_r$ avoids $\sigma$. 
Suppose for the sake of contradiction that $\pi'$ has an occurrence of $\sigma$, and note again that such an occurrence must use the letter $j$. We deduce that this occurrence does not use any letter in $\pi_1\dots\pi_{i-1}$, since these are smaller than $j$ but $\sigma$ starts with its largest letter. Thus, $j$ is the first (and hence largest) letter in the occurrence, implying that the occurrence does not use the letter $\pi_i$, since $\pi_i>j$. Now, replacing $j$ with $\pi_i$ in the occurrence would produce an occurrence of $\sigma$ in $\pi$, which is a contradiction.
\end{proof}

To generalize Proposition~\ref{leftmost-pure} to extremal-endpoint patterns, we will first show that applying certain symmetries to the pattern we obtain isomorphic complexes.
For a word $\pi=\pi_1\pi_2\dots\pi_r$ define its reversal as $\pi^R=\pi_r\dots\pi_2\pi_1$. When $\pi$ is an injective word over the alphabet $[n]$, we define its complement $\pi^C$ by replacing each entry $i$ with $n+1-i$. For a pattern $\sigma\in\S_k$, we define its complement $\sigma^C$ by replacing each entry $i$ with $k+1-i$.

\begin{prop}\label{isomorph-one}
For any pattern $\sigma$, there is an isomorphism between $\I_n^\sigma$ and $\I_n^{\sigma^R}$.
\end{prop}

\begin{proof}
Map each word $\pi$ to its reversal $\pi^R$, and note that $\pi$ avoids $\sigma$ if and only if $\pi^R$ avoids $\sigma^R$.
\end{proof}

\begin{prop}\label{isomorph-two}
For any pattern $\sigma$, there is an isomorphism between $\I_n^\sigma$  and $\I_n^{\sigma^C}$.
\end{prop}

\begin{proof}
Map each word $\pi$ to its complement $\pi^C$, and note that $\pi$ avoids $\sigma$ if and only if $\pi^C$ avoids $\sigma^C$. This isomorphism is equivalent to replacing $i$ with $n+1-i$ for each letter $i$ in our alphabet $[n]$.
\end{proof}

We can now use these symmetries to derive the following generalization of Proposition~\ref{leftmost-pure},

\begin{cor}\label{cor:first-pattern-group}
For any extremal-endpoint pattern $\sigma$, all of the maximal faces in $\I_n^{\sigma }$ are permutations in $\S_n$, hence $\I_n^{\sigma }$ is pure of dimension $n-1$. 
\end{cor}

\begin{proof}
    Every extremal-endpoint pattern can be obtained by applying reversal and/or complementation to a pattern that starts with its largest letter. Thus, the result follows from Proposition~\ref{leftmost-pure} using 
    Propositions~\ref{isomorph-one} and~\ref{isomorph-two}.
\end{proof}

\subsection{Patterns that do not contain both a consecutive ascent and a consecutive descent}
Next we consider another large class of patterns. The purity results in this section will be used in other shellability proofs later in the paper.

Define a {\em consecutive ascent} in a permutation pattern to be an ascent of the form $i,i+1$ for some $i$, that is, a pair of increasing entries having consecutive positions and consecutive values.
Similarly, define a {\em consecutive descent} in a permutation to be an descent of the form $i+1,i$ for some $i$.

\begin{prop}\label{prop:second-pattern-group}
For any pattern $\sigma$ either having no consecutive ascent or having no consecutive descent, all of the maximal faces in $\I_n^{\sigma }$ are permutations in $\S_n$, hence $\I_n^{\sigma}$ is pure of dimension $n-1$.
\end{prop}

\begin{proof}
    Let $\pi=\pi_1\dots\pi_r$ be a $\sigma$-avoiding injective word  of length $r<n$ over the alphabet $[n]$. We will show that, for any $j\in[n]\setminus\{\pi_1,\dots,\pi_r\}$, it is possible to insert $j$ into $\pi$ to create another $\sigma$-avoiding injective word.

    Let $\pi_i$ be a letter in $\pi$ with property that $\pi$ contains no other letters whose value is strictly between $\pi_i$ and $j$. Let $\pi'=\pi_1\dots \pi_{i-1} j \pi_i\dots \pi_r$.
    Since $\pi$ avoids $\sigma$, any occurrence of $\sigma$ in $\pi'$ must use the letter $j$. In fact, it must also use the letter $\pi_i$; otherwise, replacing $j$ with $\pi_i$ in the occurrence would produce an occurrence of $\sigma$ in $\pi$. 
    But having both $j$ and $\pi_i$ in the occurrence forces the pattern $\sigma$ to have a consecutive ascent (if $j<\pi_i$) or a consecutive descent (if $j>\pi_i$). 
    A similar argument shows that if 
    $\pi''=\pi_1\dots \pi_{i} j \pi_{i+1}\dots \pi_r$ contains $\sigma$, then $\sigma$ must have a 
    consecutive descent (if $\pi_i>j$) or a consecutive ascent
    (if $\pi_i<j$). 

    Thus, for both $\pi'$ and $\pi''$ to contain $\sigma$, this pattern would need have both a consecutive ascent and a consecutive descent, since this is true both for $\pi_i < j$ and for $\pi_i>j$. We deduce that at least one $\pi'$ and $\pi''$ avoids~$\sigma$.
\end{proof}

We will be especially interested in the following corollary.

\begin{cor}\label{cor:three-and-four-pure}
Each of the complexes $\I_n^{3412}$, $\I_n^{3142}$, $\I_n^{2413}$ and $\I_n^{2143}$ is pure of dimension $n-1$.  Thus, for any $\sigma $ of length at most  4, the complex $\I_n^{\sigma }$ is pure of dimension $n-1$.
\end{cor}

\begin{proof}
Each pattern of length at most 4 other than those in the set $\{ 2143, 2413, 3142, 3412 \} $ is handled by Corollary~\ref{cor:first-pattern-group}.  These other four patterns are each handled by Proposition~\ref{prop:second-pattern-group}.
\end{proof}

It is easy to generalize Proposition~\ref{prop:second-pattern-group} to collections of patterns as follows.

\begin{prop}\label{prop:second-pattern-group-collection}
Let $\Sigma$ be a collection of patterns satisfying either (or both) of these conditions: 
\begin{enumerate}
    \item no $\sigma\in\Sigma$ has a consecutive ascent, 
    \item no $\sigma\in\Sigma$ has a consecutive descent.     
\end{enumerate}
Then the complex $\I_n^{\Sigma}$ is pure of dimension $n-1$.
\end{prop}

\begin{proof}
    We use the same idea as in the proof of Proposition~\ref{prop:second-pattern-group}. 
    Let $\pi=\pi_1\dots\pi_r$ be a $\Sigma$-avoiding injective word  of length $r<n$ over the alphabet $[n]$, let $j\in[n]\setminus\{\pi_1,\dots,\pi_r\}$, and let $\pi'$ and $\pi''$ be defined as in the proof of Proposition~\ref{prop:second-pattern-group}.

    Suppose first that $j<\pi_i$. If $\Sigma$ satisfies condition (1), then $\pi'$ must avoid every $\sigma\in\Sigma$, because the only way for $\sigma$ to occur in $\pi'$ would be if $\sigma$ had a consecutive ascent. Similarly, if $\Sigma$ satisfies condition (2), then  $\pi''$ must avoid every $\sigma\in\Sigma$.

    A similar argument applies if $j>\pi_i$, but now the roles of $\pi'$ and $\pi''$ are reversed. In all cases, we deduce that at least one of the words $\pi'$ or $\pi''$ is $\Sigma$-avoiding.
\end{proof}

The following consequence of Proposition~\ref{prop:second-pattern-group-collection} will be used later as an important ingredient in our proof of shellability for $\I_n^{2413,3142}$.

\begin{cor}\label{cor:separable-pure}
The complex $\I_n^{2413,3142}$ of separable injective words over the alphabet $[n]$ is pure of dimension $n-1$.
\end{cor}

\begin{proof}
    The set $\Sigma=\{2413,3142\}$ satisfies the conditions of Proposition~\ref{prop:second-pattern-group-collection}.
\end{proof}

\section{Shelling the complex of injective words avoiding an extremal-endpoint pattern}\label{shell-subcomplex}

In this section we prove shellability of $\I_n^\sigma$ for extremal-endpoint patterns $\sigma$, that is, patterns that start or end with  their smallest or largest letter, as introduced in Section~\ref{sec:extremal-endpoint}.
We will show that $\I_n^\sigma$ is shellable for such patterns.

We first consider patterns that start with their largest letter.
We will prove that for any $\sigma\in\S_k$ with $\sigma_1=k$, the complex $\I_n^\sigma$ is shellable with respect to the lexicographic order.
We start by characterizing the blockers in the complex $\I_n^\sigma$ for such a pattern.

\begin{lem}\label{lem:blockers}
    Let $\sigma\in\S_k$ with $\sigma_1=k$, and let $\pi\in\S_n(\sigma)$. The entry $\pi_j$ is a $\sigma$-blocker with respect to $\lex$ if and only if it is an unrestricted blocker with respect to $\lex$, i.e., it satisfies conditions (a) or (b) from Lemma~\ref{lem:unrestricted_blockers}.
\end{lem}

\begin{proof}
By Lemma~\ref{lem:no_gained_blockers}, every $\sigma$-blocker is also an unrestricted blocker. Thus, it suffices to show that every unrestricted blocker (equivalently, by Lemma~\ref{lem:unrestricted_blockers}, any $\pi_j$ satisfying conditions (a) or (b)) is also a $\sigma$-blocker.

Suppose first that $\pi_j$ satisfies condition (a), and let $\pi_i$ be the leftmost entry such that $\pi_i>\pi_j$. Let
$\pi'=\pi_1\dots\pi_{i-1}\pi_j\pi_i\pi_{i+1}\dots\widehat{\pi_j}\dots\pi_n$ be the permutation obtained by moving $\pi_j$ to position $i$. Then $\pi'\lex\pi$, so it suffices to show that $\pi'$ avoids $\sigma$.
Since $\pi_1\dots \widehat{\pi_j} \dots \pi_n$ avoids $\sigma$, the only way for $\pi'$ to have an occurrence of $\sigma$ is if the letter $\pi_j$ that has been moved is part of this occurrence. In this case, since $\pi_j$ is larger than all letters to its left in $\pi'$, and $\sigma$ starts with its largest letter, $\pi_j$ must be the first letter of the occurrence. Since $\pi_j < \pi_{i}$, the letter $\pi_{i}$ cannot be part of the occurrence, and in fact we could replace $\pi_j$ with $\pi_{i}$ to find another occurrence of $\sigma$ within $\pi'$. But then this occurrence would already have been present in $\pi$, a contradiction.

Suppose now that $\pi_j$ does not safisfy condition (a) but it satisfies condition (b), i.e., $\pi_j$ is a record and $\pi_j>\pi_{j+1}$.
Let $\pi' = \pi_1\dots \pi_{j-1}\pi_{j+1}\pi_j\pi_{j+1}\dots \pi_n$ be obtained by swapping $\pi_j$ and $\pi_{j+1}$. Then $\pi'\lex\pi$, so it suffices to show that $\pi'$ avoids $\sigma$.  Since $\pi$ avoids $\sigma$, the only way for $\pi'$ to have an occurrence of $\sigma$ is if both $\pi_{j+1}$ and $\pi_j$ are part of the occurrence. But this is impossible because $\sigma$ starts with its largest letter, whereas $\pi_j$ is a record in $\pi'$ (since it is a record in $\pi$ and $\pi_j>\pi_{j+1}$).
\end{proof}

When $\sigma$ starts with its largest letter, Lemma~\ref{lem:blockers} states that an entry of $\pi$ is a $\sigma$-blocker with respect to $\lex$ if and only if it is not an ephemeral record. Thus, the number of $\sigma$-blockers of $\pi\in\S_n(\sigma)$ is $n-\er(\pi)$ in this case.

\begin{thm}\label{thm:shellability_start_with_k}
For any pattern $\sigma\in\S_k$ with $\sigma_1=k$, the complex $\I_n^\sigma$ is pure of dimension $n-1$ and is shellable with respect to $\lex$.   
 Moreover, 
$$h_i(\I_n^\sigma)=|\{\pi\in\S_n(\sigma):\er(\pi)=n-i\}|,$$
and so
$$h(\I_n^\sigma,t)=\sum_{\pi\in\S_n(\sigma)} t^{n-\er(\pi)}.$$
In particular, the homology of $\I_n^\sigma$ is concentrated in top degree, and its rank is
$$h_n(\I_n^\sigma)=|\L_n(\sigma )|=(-1)^{n-1}\tilde{\chi}(\I_n^{\sigma }).$$ 
\end{thm}

\begin{proof}
The fact that $\I_n^{\sigma }$ is pure of dimension $n-1$ is proven in Proposition~\ref{leftmost-pure}.
By Lemma~\ref{lem:blockers} the $\sigma$-blockers in any $\sigma$-avoiding permutation with respect to $\lex$ are exactly the unrestricted blockers.   
This allows us to invoke Theorem~\ref{thm:shelling-via-unrestricted-blockers} to deduce the desired shellability result. 

The statement about $h$-numbers follows from Lemma~\ref{h-via-blockers}, using the fact that the number of $\sigma$-blockers of $\pi\in\S_n(\sigma)$ with respect to $\lex$ is $n-\er(\pi)$.
\end{proof}

\begin{ex}
Let $\sigma=312$. For $n=3$, ordering the permutations in $\S_3(312)$ lexicographically,  we have the shelling order
$123, 1\d{3}\d{2}, \d{2}\d{1}3, 2\d{3}\d{1}, \d{3}\d{2}\d{1}$ on the facets of $\I_3^{312}$.   The dots underneath indicate which letters are blockers.

For $n=4$, the lexicographic shelling order of $\S_4(312)$ is
$$1234, 12\d{4}\d{3},1\d{3}\d{2}4, 13\d{4}\d{2}, 1\d{4}\d{3}\d{2}, \d{2}\d{1}34, \d{2}\d{1}\d{4}\d{3}, 2\d{3}\d{1}4, 23\d{4}\d{1}, 2\d{4}\d{3}\d{1},  \d{3}\d{2}\d{1}4, \d{3}\d{2}\d{4}\d{1},3\d{4}\d{2}\d{1}, \d{4}\d{3}\d{2}\d{1}.$$
\end{ex}

Recall that $\L_n(\sigma)$ denotes the set of $\sigma$-avoiding permutations in $\S_n$ no ephemeral records.

\begin{cor}\label{homology-facet-description}
For any pattern $\sigma\in\S_k$ with $\sigma_1=k$, the homology facets of the shelling for $\I_n^{\sigma }$ given by the lexicographic order
are exactly those permutations in $\L_n(\sigma)$.
 \end{cor}

 \begin{cor}\label{Euler-char}
For any pattern $\sigma\in\S_k$ with $\sigma_1=k$, we have 
$\tilde{\chi }(\I_n^{\sigma }) = (-1)^{n-1}|\L_n(\sigma )|$. 
 \end{cor}

 \begin{proof}
Using Equation~\eqref{eq:chi-h} and setting $i=n$ in Theorem~\ref{thm:shellability_start_with_k}, we get
$\tilde{\chi }(\I_n^{\sigma }) = (-1)^{n-1}h_n(\I_n^{\sigma }) = (-1)^{n-1}|\L_n(\sigma )|$. 
 \end{proof}

For patterns that start with their largest letter, shellability of $\I_n^\sigma$, together with our characterization of $\sigma$-blockers with respect to $\lex$, allows us to refine the enumeration of $\sigma$-avoiding permutation by the number of ephemeral records.

\begin{cor}\label{cor:refined_enumeration}
    For any pattern $\sigma\in\S_k$ with $\sigma_1=k$, we have
     $$\sum_{n\ge0}\sum_{\pi\in\S_n(\sigma)}t^{\er(\pi)}x^n
     =\sum_{n=0}^\infty \bh(\I_n^\sigma,t)\,x^n
=\frac{1}{1-(t-1)x}\,P_\sigma\left(\frac{x}{1-(t-1)x}\right).
     $$
     In particular, setting $t=0$,
     $$\sum_{n\ge0}\card{\L_n(\sigma)}\,x^n
     =\sum_{n=0}^\infty h_n(\I_n^\sigma)\,x^n
=\frac{1}{1+x}\,P_\sigma\left(\frac{x}{1+x}\right).
     $$
\end{cor}

\begin{proof}
    This follows by combining Theorem~\ref{thm:shellability_start_with_k} and Proposition~\ref{prop:GFh}.
\end{proof}

We can now apply symmetries to deduce shellability of complexes of injective words avoiding any extremal-endpoint pattern.

\begin{cor}\label{cor:extremal-endpoint}
For any pattern $\sigma$ that starts or ends with its smallest or its largest letter, the complex $\I_n^\sigma$ is shellable.
\end{cor}

\begin{proof}
This follows from Theorem~\ref{thm:shellability_start_with_k} by applying the symmetries in Propositions~\ref{isomorph-one} and~\ref{isomorph-two}.  
\end{proof}

 \section{Patterns of length 3}\label{sec:length3}

Since every pattern of length 3 is an extremal-endpoint pattern,  Corollary~\ref{cor:extremal-endpoint} implies that the corresponding complexes are shellable.
The isomorphisms mentioned in the next result are given by Propositions~\ref{isomorph-one} and~\ref{isomorph-two}.

\begin{cor}\label{cor:length3}
For every $n\ge0$ and every pattern $\sigma\in\S_3$, the complex $\I_n^{\sigma}$ is shellable. 
Additionally, we have the complex isomorphisms $\I_n^{123}\cong\I_n^{321}$ and $\I_n^{132}\cong\I_n^{213}\cong\I_n^{231}\cong\I_n^{312}$.
\end{cor}

Theorem~\ref{thm:shellability_start_with_k} implies, specifically, that the lexicographic order is a shelling order for both 
$\I_n^{321}$ and $\I_n^{312}$. Using Proposition~\ref{isomorph-one}, it follows that the lexicographic order on the reversal of the words is a shelling order for $\I_n^{123}$ and $\I_n^{213}$.
Applying Proposition~\ref{isomorph-two}, a shelling order for $\I_n^{132}$ is given by the lexicographic order on words where the alphabet is ordered $n<n-1<\dots<2<1$. For example, for $n=3$, the ordering of the facets is $321,312,231,213,123$.
Finally, by Proposition~\ref{isomorph-one}, this same ordering applied to the reversal of the words is a shelling order for $\I_n^{231}$.
For example, for $n=3$, the ordering of the facets is $123,213,132,312,321$.

It is natural to ask whether the complexes $\I_n^{321}$ and $\I_n^{312}$ are isomorphic, and more generally, whether $\I_n^{\sigma}\cong\I_n^{\tau}$ whenever $\sigma$ and $\tau$ are Wilf-equivalent patterns --- after all, we know by Corollary~\ref{cor:refined_enumeration} that these two complexes have the same $h$-numbers. It turns out that this is not the case.

\begin{prop}\label{prop:not_isomorphic}
    For $n\ge4$, the complexes $\I_n^{321}$ and $\I_n^{312}$ are not isomorphic.
\end{prop}

\begin{proof}
    Both complexes have $n(n-1)$ faces of dimension $1$, which are all the words of the form $ab$ with $a,b\in[n]$ and $a\neq b$.

    Let us count how many $321$-avoiding injective words over $[n]$ of length 3 contain a given word $ab$. If $a<b$, this number is $3n-6$, since each of the $n-2$ letters in $[n]\setminus\{a,b\}$ can be inserted in any of the $3$ available positions without creating an occurrence of $321$. If $a>b$, this number is $2n-4$, because each letter can be inserted in $2$ positions, noting that there is always one position that creates a pattern $321$. 

    On the other hand, the number of $312$-avoiding injective words over $[n]$ of length 3 that contain the word $13$ is $3+2(n-3)=2n-3$.
    Indeed, the letter 2 can be inserted in any of the 3 available positions, whereas any letter $c\ge 4$ can be inserted in 2 positions (any but the first) without creating a pattern $312$. For $n\ge 4$, the number $2n-3$ is not equal to either $3n-6$ or $2n-4$, so the face posets of $\I_n^{321}$ and $\I_n^{312}$ are not isomorphic.
\end{proof}

In the rest of this section we discuss some enumerative results for patterns of length 3 and the associated pattern-avoiding complexes of injective words. We will give two different derivations of the generating function for the $h$-polynomials of these complexes. The first one will rely on the following enumeration of $312$-avoiding permutations and $321$-avoiding permutations by the number of ephemeral records.

\begin{lem}\label{lem:Riordan}
Let $\sigma\in\{312,321\}$. The generating function for $\sigma$-avoiding permutations with respect to the number of ephemeral records is
\begin{equation}\label{eq:GF312}\sum_{n=0}^\infty \sum_{\pi\in\S_n(\sigma)} t^{\er(\pi)} x^n
=\frac{2}{1-(t-1)x+\sqrt{(1-(t-1)x)(1-(t+3)x)}}
.
\end{equation}
In particular, setting $t=0$, we have
$|\L_n(\sigma)|=R_n$ (the $n$th Riordan number as defined in Section~\ref{sec:permutations}) for all $n$.
\end{lem}

\begin{proof}
Let us first focus on $312$-avoiding permutations. We use a bijection between $\S_n(312)$ and the set $\D_n$ of Dyck paths of with $2n$ steps, which is due to Krattenthaler~\cite{Kratt}, up to trivial symmetries. To describe the path $D$ corresponding to $\pi\in \S_n(312)$, first draw the permutation diagram of $\pi$, which is the $n\times n$ grid where the $i$th column from the left has a cross in row $\pi_i$, with the convention that rows are numbered starting from the bottom. Assign coordinates to the vertices of the grid, so that the bottom-left corner is at the origin, and the top-right corner has coordinates $(n,n)$.
Let $D$ be the lattice path from $(0,0)$ to $(n,n)$ with steps $N=(0,1)$ and $E=(1,0)$ that passes above all the crosses but stays as close to the diagonal as possible; equivalently, the path whose peaks $NE$ occur precisely around the cells containing the left-to-right maxima of $\pi$. See the left of Figure~\ref{fig:312Dyck} for an example, where $\pi=34265187$.

\begin{figure}[h]
\centering
\begin{tikzpicture}[scale=0.5]
 \draw (0,0) grid (8,8);
 \draw[dotted] (0,0)--(8,8);
 \perm{{3,4,2,6,5,1,8,7}}
 \draw[red,ultra thick] (0,0) \N\N\N\E\N\E\E\N\N\E\E\E\N\N\E\E;
 \begin{scope}[shift={(12,0)}]
  \draw (0,0) grid (8,8);
 \draw[dotted] (0,0)--(8,8);
 \perm{{3,4,1,6,2,5,8,7}}
 \draw[red,ultra thick] (0,0) \N\N\N\E\N\E\E\N\N\E\E\E\N\N\E\E;
 \end{scope}
\end{tikzpicture}
\caption{Examples of the bijections between $312$-avoiding permutations and Dyck paths (left), and between $321$-avoiding permutations and Dyck paths (right).}
\label{fig:312Dyck}
\end{figure}
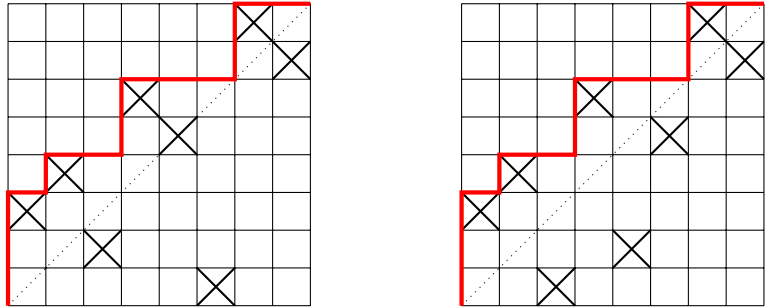

To describe the inverse map, let $D$ by a Dyck path, drawn as a path from $(0,0)$ to $(n,n)$ with steps $N$ and $E$ that stays weakly above the diagonal. Start by placing crosses in the cells inside the peaks $NE$ of the path. Next, for each column of the grid that does not have a cross yet, from left to right, place a cross as high as possible so that it is below $D$ and it is not in a row that already contains a cross. This produces the diagram of a permutation $\pi\in \S_n(312)$.

Via this bijection $\pi\mapsto D$, a record of $\pi$ is followed by a smaller entry if and only if the corresponding peak $NE$ is followed by an $E$ step. Thus, an ephemeral record corresponds to a peak $NE$ that is not followed by an $E$ step (i.e., it is either followed by a $N$ step or at the end of the path). Define $\er(D)$ to be the number of peaks with this property, and let
$$R(t,x)=\sum_{n=0}^\infty \sum_{\pi\in\S_n(312)} t^{\er(\pi)} x^n=\sum_{n=0}^\infty \sum_{D\in\D_n} t^{\er(D)} x^n$$
be the generating function that we want to compute.

We claim that $R(t,x)$ satisfies the equation
\begin{equation}\label{eq:R}
R(t,x)=1+x\left(R(t,x)+(t-1)+xR(t,x)(1-t)\right)R(t,x).
\end{equation}
To see this, recall that every nonempty Dyck path $D$ can be uniquely decomposed as $D=N D_1 E D_2$, where $D_1$ and $D_2$ are Dyck paths. 
The peaks $NE$ that are not followed by a $E$ in $D$ are those coming from either $D_1$ or $D_2$, with the following exceptions: if $D_1$ is empty, we get an extra such peak (this accounts for the correction term $t-1$); and if $D_1$ ends with a peak, this peak was not followed by a $E$ in $D_1$ but it is in $D$ (this accounts for the correction term $xR(t,x)(1-t)$).

Solving Equation~\eqref{eq:R}, we obtain 
$$R(t,x)=\frac{1-(t-1)x-\sqrt{1-2(t+1)x+(t^2+2t-3)x^2}}{2x(1-(t-1)x)},$$
which is equivalent to the stated formula.

Permutations in $\S_n(321)$ can be counted similarly, by using another well-known bijection between $\S_n(321)$ and Dyck paths, see e.g.~\cite{Elizalde}. To find the image of $\pi\in \S_n(321)$, we again draw the permutation diagram of $\pi$, and let $D$ be the lattice path from $(0,0)$ to $(n,n)$ with steps $N$ and $E$ that leaves all the crosses below but stays as close to the diagonal as possible. See the right of Figure~\ref{fig:312Dyck} for an example, where $\pi=34162587$.

The inverse map is now described as follows: given a Dyck path $D$, again start by placing crosses in the cells inside its peaks. Then, for each column of the grid that does not have a cross yet, from left to right, place a cross as {\em low} as possible so that it is not in a row that already contains a cross. This produces the diagram of a permutation $\pi\in \S_n(321)$.

As before, a record of $\pi$ is followed by a smaller entry if and only if the corresponding peak $NE$ is followed by an $E$ step. Therefore, $\er(\pi)=\er(D)$, and we get the same generating function $R(t,x)$.

Finally, setting $t=0$ in Equation~\eqref{eq:GF312}, and recalling that $\L_n(\sigma)$ is the set of $\sigma$-avoiding permutations with no ephemeral records, we obtain
$$\sum_{n=0}^\infty |\L_n(\sigma)|\,x^n= \frac{2}{1+x+\sqrt{(1+x)(1-3x)}},$$
which is the generating function from Equation~\eqref{eq:Riordan}.
\end{proof}

\begin{prop}\label{Riordan-as-Euler-char}
For any $\sigma\in\S_3$, we have
$$ \sum_{n=0}^\infty h(\I_n^{\sigma},t)\,x^n=\frac{2}{1-(1-t)x+\sqrt{(1-(1-t)x)(1-(1+3t)x)}}.$$
Additionally, the homology of $\I_n^\sigma$ is concentrated in top degree, and its rank is
$h_n(\I_n^{\sigma})=|\L_n(\sigma )|=(-1)^{n-1}\tilde\chi(\I_n^{\sigma})=R_n$ (the $n$th Riordan number) for all $n$.
\end{prop}

\begin{proof}
By Corollary~\ref{cor:length3}, the complexes $\I^\sigma_n$ for $\sigma\in\S_3$ are all isomorphic to either
$\I^{312}_n$ or $\I^{321}_n$, so in the rest of this proof we assume that  $\sigma\in\{312,321\}$.

We give two derivations of the generating function for the $h$-polynomials. The first one uses shellability and the interpretation of the $h$-numbers as counting facets by the number of blockers. By Theorem~\ref{thm:shellability_start_with_k}, we have 
$$h(\I_n^{\sigma},t)=\sum_{\pi\in\S_n(\sigma)} t^{n-\er(\pi)}\quad\text{and}\quad
\bh(\I_n^{\sigma},t)=\sum_{\pi\in\S_n(\sigma)} t^{\er(\pi)}.$$ 
Thus, Lemma~\ref{lem:Riordan} implies that
$$
\sum_{n=0}^\infty \bh(\I_n^{\sigma},t)\,x^n=\frac{2}{1-(t-1)x+\sqrt{(1-(t-1)x)(1-(t+3)x)}},
$$
and making the substitution $(t,x)\mapsto(1/t,tx)$ yields the stated generating function for $\sum_{n=0}^\infty h(\I_n^{\sigma},t)\,x^n$.

The second proof is based on the relationship between $h$-numbers and $f$-numbers. 
Using Proposition~\ref{prop:GFh} and the fact that $P_{\sigma}(x)=C(x)$, we get 
\[
    \sum_{n=0}^\infty h(\I_n^{\sigma},t)\,x^n
=\frac{1}{1-(1-t)x}\,C\left(\frac{tx}{1-(1-t)x}\right)
=\frac{1-\sqrt{\frac{1-(1+3t)x}{1-(1-t)x}}}{2tx},
\]
which is equivalent to the stated expression. Additionally, 
$$\sum_{n=0}^\infty h_n(\I_n^{\sigma})\,x^n=
\frac{1}{1+x}\,C\left(\frac{x}{1+x}\right)=\frac{2}{1+x+\sqrt{(1+x)(1-3x)}},$$
which is the generating function for the Riordan numbers given in Equation~\eqref{eq:Riordan}.
\end{proof}

\section{Patterns of length 4}\label{sec:length4}
Corollary~\ref{cor:extremal-endpoint} implies that the complex $\I_n^\sigma$ is shellable for all patterns $\sigma\in\S_4$ except 
$3412,  2143, 3142, 2413$.
In this section we handle the pattern $3412$ and its reversal $2143$. 

\subsection{The pattern $3412$} \label{sec:3412}
Our goal is to prove that the complex $\I_n^{3412}$ is shellable with respect to the lexicographic order.
We start by giving some sufficient conditions for $3412$-blockers.

\begin{lem}\label{lem:3412blockers-sufficient}
    Let $\pi\in\S_n(3412)$. If the entry $\pi_j$ satisfies any of the following four conditions, then it is a $3412$-blocker of $\pi$ with respect to $\lex$:
    \begin{enumerate}[(i)]
        \item $\pi_j$ is a 1 in some occurrence of $231$ in $\pi$,
        \item $\pi_j$ is a 3 in some occurrence of $312$ in $\pi$,
        \item $j<n$ and $\pi_j>\pi_{j+1}$,
        \item $j>1$ and $\pi_{j-1}>\pi_{j}$.
    \end{enumerate}
\end{lem}

\begin{proof}
    Suppose first that (i) is satisfied.  Let $\pi_{i_1}\pi_{i_2}\pi_j$ be an occurrence of $231$ that minimizes $i_2$. Note that $i_1<i_2<j$ and $\pi_j <\pi_{i_1} < \pi_{i_2}$. Since $\pi_j<\pi_{i_2}$, moving $\pi_j$ to position $i_2$ creates a lexicographically smaller permutation. Let us show that it avoids $3412$, by arguing that $\pi_j$ cannot be any of the elements of an occurrence of this pattern:
    \begin{itemize}
        \item If $\pi_j$ was the 3 of such an occurrence, say $\pi_j\pi_a\pi_b\pi_c$, then the fact that 
        $\pi_b < \pi_c < \pi_j < \pi_a$, combined with $\pi_j < \pi_{i_1} < \pi_{i_2}$, implies that $\pi_b<\pi_c<\pi_j<\pi_{i_1}$. But then $\pi_{i_1}\pi_{i_2}\pi_b\pi_c$ would have been an occurrence of $3412$ in $\pi$, since $i_1 < i_2 < j < b < c$. 
        \item If $\pi_j$ was the $4$ of such an occurrence, then replacing it with $\pi_{i_2}$ would give an occurrence of $3412$ in $\pi$.
        \item By the choice of $i_2$, there are no two increasing entries to the left of position $i_2$ that are larger than $\pi_j$, so $\pi_j$ cannot be the 1 or the 2 of such an occurrence.
    \end{itemize}

    Similarly, if (ii) is satisfied, let $\pi_j\pi_{i_3}\pi_{i_4}$ be an occurrence of $312$ that maximizes $i_3$. A symmetric argument shows that moving $\pi_j$ to position $i_3$ produces a permutation that avoids $3412$. Additionally, this permutation is lexicographically smaller than $\pi$ because $\pi_j>\pi_{j+1}$. Indeed, if we had $\pi_j<\pi_{j+1}$, then $\pi_j\pi_{j+1}\pi_{i_3}\pi_{i_4}$ would have been an occurrence of $3412$ in~$\pi$.  

    Suppose now that (iii) is satisfied. If swapping $\pi_j$ and $\pi_{j+1}$ produces a $3412$-avoiding permutation, then we are done. So we can assume that this swap creates an occurrence of $3412$, which necessarily uses both entries $\pi_{j+1}$ and $\pi_j$. 
    If $\pi_{j+1}\pi_j$ form the $12$ of such an occurrence, say $\pi_a\pi_b\pi_{j+1}\pi_j$, then $\pi_a\pi_b\pi_j$ is an occurrence of $231$ in $\pi$, so condition (i) is satisfied.
    If $\pi_{j+1}\pi_j$ form the $34$ of such an occurrence, say $\pi_{j+1}\pi_j\pi_c\pi_d$, then $\pi_j\pi_c\pi_d$ is an occurrence of $312$ in $\pi$, so condition (ii) is satisfied.

    Finally, suppose that (iv) is satisfied. If swapping $\pi_{j-1}$ and $\pi_{j}$ produces a $3412$-avoiding permutation, then we are done. So we can assume that this swap creates an occurrence of $3412$, which necessarily uses both entries $\pi_{j}$ and $\pi_{j-1}$. 
    If $\pi_{j}\pi_{j-1}$ form the $12$ of such an occurrence, then condition (i) is satisfied. If they form the $34$ of such an occurrence, then condition (ii) is satisfied.
\end{proof}

Lemma~\ref{lem:3412blockers-sufficient} can be rephrased as follows.

\begin{lem}\label{lem:3412nonblockers} 
Let $\pi\in\S_n(3412)$. Any entry $\pi_j$ that is not a $3412$-blocker of $\pi$ with respect to $\lex$ must satisfy all of these conditions:
    \begin{enumerate}[(i)]
        \item the entries to the left of $\pi_j$ that are larger than $\pi_j$ form a decreasing sequence,
        \item the entries to the right of $\pi_j$ that are smaller than $\pi_j$ form a decreasing sequence,
        \item $j=n$ or $\pi_j<\pi_{j+1}$,
        \item $j=1$ or $\pi_{j-1}<\pi_{j}$.
    \end{enumerate}
Additionally, if $\pi_i$ is one of the entries in (i) (i.e. $i<j$ and $\pi_i>\pi_j$), then moving $\pi_j$ to position $i$ creates an occurrence of $3412$ where $\pi_j$ plays the role of $3$.
\end{lem}

\begin{proof}
    The contrapositive of Lemma~\ref{lem:3412blockers-sufficient} states that any entry $\pi_j$ that is not a $3412$-blocker must satisfy (i)--(iv). Let us prove the additional statement. Since $\pi_j$ is not a $3412$-blocker and moving $\pi_j$ to position $i$ produces a lexicographically smaller permutation, this move must create an occurrence of $3412$. Let us argue that $\pi_j$ must be the 3 of such an occurrence. Clearly, $\pi_j$ cannot be the $1$ or the $2$, because then $\pi_j$ would be the $1$ in an occurrence of $231$ in $\pi$, contradicting Lemma~\ref{lem:3412blockers-sufficient}(i). And if $\pi_j$ was the $4$ of such an occurrence, then replacing it with $\pi_i$ would give an occurrence of $3412$ in $\pi$, which is a contradiction.
\end{proof}

\begin{thm}\label{thm:3412}
    The complex $\I_n^{3412}$ is pure of dimension $n-1$ and is shellable with respect to~$\lex$.
\end{thm}

\begin{proof}
    The fact that $\I_n^{3412}$ is pure of dimension $n-1$ is proven in Proposition~\ref{leftmost-pure}.
    Recall from Definition~\ref{def:blocker} that a $3412$-blocker of $\pi\in\S_n(3412)$ with respect to $\lex$ is an entry $\pi_j$ that can be moved within $\pi$ to obtain a lexicographically smaller $3412$-avoiding permutation. In this proof, such an entry will be simply called a {\em blocker}, and the other entries of $\pi$ will be called {\em non-blockers}.
    
    To prove the theorem, we will show that, for any given $\pi\in\S_n(3412)$, it is not possible to move any set of non-blockers within $\pi$ to obtain a lexicographically smaller $3412$-avoiding permutation.
    
    Suppose for the sake of contradiction that such a set of non-blockers existed, call this set $S$, and let 
    $\pi'\in\S_n(3412)$ be the resulting permutation, which satisfies $\pi'\lex\pi$. Let $i$ the leftmost position where $\pi$ and $\pi'$ disagree, and note that $\pi'_i<\pi_i$. Let $j$ be the position where $\pi'_i$ appears in $\pi$, that is, $\pi_j=\pi'_i$, and note that $j>i$.

    Suppose first that $\pi_j\notin S$. Then $\pi_i,\pi_{i+1},\dots,\pi_{j-1}\in S$, since the relative position of each of these entries with respect to $\pi_j$ is different in $\pi'$ and $\pi$. But since elements of $S$ are non-blockers, Lemma~\ref{lem:3412nonblockers}(iii) implies that $\pi_i<\pi_{i+1}<\dots<\pi_{j-1}<\pi_j$, contradicting the fact that $\pi_j<\pi_i$. We deduce that $\pi_j\in S$, and in particular $\pi_j$ is a non-blocker.

    Consider now the permutation $\tilde{\pi}=\pi_1\pi_2\dots\pi_{i-1}\pi_j\pi_i\pi_{i+1}\pi_{i+2}\dots\widehat{\pi_j}\dots\pi_n$, obtained from $\pi$ by moving $\pi_j$ to position $i$. 
    Since $\pi_j$ is a non-blocker, this move creates an occurrence of $3412$ where $\pi_j$ plays the role of 3, by Lemma~\ref{lem:3412nonblockers}. If we denote such an occurrence by $\pi_j\pi_b\pi_c\pi_d$, then $\pi_j\pi_i\pi_c\pi_d$ is also an occurrence of $3412$ in $\tilde\pi$. Let us focus on occurrences of $3412$ of the form $\pi_j\pi_i\pi_c\pi_d$. If there is more than one, take one that minimizes the value of $\pi_d$, and among these (if there is still more than one), take the one that minimizes the value of $\pi_c$. We will show that the entries $\pi_i$, $\pi_c$ and $\pi_d$ of this occurrence are all blockers of $\pi$. 

    First we show that $\pi_i$ is a blocker. Since $\pi_j$ is a non-blocker of $\pi$, Lemma~\ref{lem:3412nonblockers}(iv) implies that $\pi_{j-1}<\pi_j$. Thus, $\pi_i\pi_{j-1}\pi_j$ is an occurrence of $312$ in $\pi$, and Lemma~\ref{lem:3412blockers-sufficient}(ii) implies that $\pi_i$ is a blocker of $\pi$.

    Next we show that $\pi_c$ is a blocker. If $\pi_{c-1}<\pi_c$, then $\pi_j\pi_i\pi_{c-1}\pi_d$ would be an occurrence of $3412$ in $\tilde\pi$ with a smaller value playing the role of $1$, contradicting the choice of $\pi_c$. Thus, $\pi_{c-1}>\pi_c$. By Lemma~\ref{lem:3412blockers-sufficient}(iv), this implies that $\pi_c$ is a blocker of $\pi$.

    It remains to show that $\pi_d$ is a blocker of $\pi$. Suppose that this is not the case. Then, by Lemma~\ref{lem:3412nonblockers}, the permutation $\pi_1\pi_2\dots\pi_{i-1}\pi_d\pi_i\pi_{i+1}\pi_{i+2}\dots\widehat{\pi_d}\dots\pi_n$, obtained from $\pi$ by moving $\pi_d$ to position $i$ would contain a $3412$ where $\pi_d$ plays the role of~3. Let $\pi_d\pi_e\pi_f\pi_g$ be such an occurrence. Then $\pi_f<\pi_g<\pi_d$, so
    $\pi_j\pi_i\pi_f\pi_g$ would be an occurrence of $3412$ in $\tilde\pi$ where the value of the 2 would be smaller than $\pi_d$, contradicting the choice of $\pi_d$.

    This proves that $\pi_i$, $\pi_c$ and $\pi_d$ are all blockers of $\pi$. It follows that the relative order of the entries $\pi_j\pi_i\pi_c\pi_d$ is the same in $\pi'$ as in $\tilde\pi$. In particular, these entries form an occurrence of the pattern $3412$ in $\pi'$, which is a contradiction.   
\end{proof}

We can now combine our results to show shellability for most patterns of length~4.

\begin{prop}\label{cor:shellable-length4}
For every pattern $\sigma\in \S_4\setminus\{3142, 2413\}$, the complex $\I_n^{\sigma}$ is shellable.
\end{prop}

\begin{proof}
If $\sigma\in\S_4$ is an extremal-endpoint pattern, shellability of $\I_n^{\sigma}$ follows from Corollary~\ref{cor:extremal-endpoint}.
Any $\sigma\in\S_4$ that is not an extremal-endpoint patterns must start and end with a $2$ or a $3$, that is, $\sigma\in\{3412,  2143, 3142, 2413\}$.
Shellability of $\I_n^{3412}$ is proved in Theorem~\ref{thm:3412}, and 
shellability of $\I_n^{2143}$ follows from Proposition~\ref{isomorph-one}.
\end{proof}

For the pattern $3142$ (and its reversal $2413$), shellability of the corresponding pattern-avoiding complex is an open question.

\subsection{Refined enumeration of permutations avoiding patterns of length~4}
\label{sec:h-length4}

In this section we apply Corollary~\ref{cor:refined_enumeration} to obtain enumerative results for permutations avoiding patterns $\sigma\in\S_4$ with $\sigma_1=4$. We use shellability of $\I^\sigma_n$ to deduce that the number of $\sigma$-avoiding permutations with no ephemeral records is equal (up to sign) to the reduced Euler characteristic of this complex and, more generally, to obtain a bivariate generating function counting $\sigma$-avoiding permutations according to the size and the number of ephemeral records.

It is known~\cite{Stankova,StankovaWest,BWX} that the 24 patterns in $\S_4$ fall into three Wilf-equivalence classes:
\begin{enumerate}[(i)]
\item $1342,2431,3124,4213,2413,3142,2314,4132,1423,3241$ (enumerated by B\'ona~\cite{Bona}),
\item $1234,4321,1243,3421,4312,2134,1432,2341,3214,4123,2143,3412$ (enumerated by Gessel~\cite{Gessel}),
\item $1324,4231$ (no known enumeration formula).
\end{enumerate}

For patterns $\sigma$ in class (i), it was shown by B\'ona~\cite{Bona} that 
\begin{equation}\label{eq:bona}
P_{\sigma}(x)=\frac{32x}{1+20x-8x^2-(1-8x)^{3/2}}.
\end{equation}
Applying Corollary~\ref{cor:refined_enumeration} to the patterns $4213$ and $4132$, which are in class (i), we obtain the following refined enumeration.

\begin{thm}\label{thm:4213enumeration}
    For $\sigma\in\{4213,4132\}$, we have
\begin{multline*}
    \sum_{n\ge0}\sum_{\pi\in\S_n(\sigma)}t^{\er(\pi)}x^n
=    \sum_{n=0}^\infty \bh(\I_n^{\sigma},t)\,x^n\\
=\frac{32x}{1-2(t-11)x+(t^2-22t+13)x^2-\sqrt{(1-(t-1)x)(1-(t+7)x)^3}}.
\end{multline*}
In particular, setting $t=0$,
\[
    \sum_{n=0}^\infty \card{\L_n(\sigma)}\,x^n=
    \sum_{n=0}^\infty h_n(\I_n^{\sigma})\,x^n
=\frac{32x}{1+22x+13x^2-\sqrt{(1-x)(1-7x)^3}}.
\]
\end{thm}

The polynomials $\bh(\I_n^{4213},t)=\bh(\I_n^{4132},t)$ for small values of $n$ are given in Table~\ref{tab:h4213}. The sequence of constant terms (Table~\ref{tab:LSR}), giving the reduced Euler characteristic (up to sign), does not currently appear in~\cite{OEIS}.
A surprising consequence of this topological approach is that the number of ephemeral records has the same distribution on $\S_n(4213)$ and $\S_n(4132)$.

\begin{table}[htb]
    $$
    \begin{array}{r|c}
       n  & \bh(\I_n^{4213},t)=\bh(\I_n^{4132},t) \\ \hline
    1 & t \\ 2 & t^2+1 \\ 3 &  t^3+3t+2 \\ 4 & t^4 + 6t^2 + 8t + 8\\
    5 & t^5 + 10t^3 + 20t^2 + 40t + 32 \\
    6 & t^6 + 15t^4 + 40t^3 + 120t^2 + 192t + 144\\
    7 & t^7 + 21t^5 + 70t^4 + 280t^3 + 672t^2 + 1008t + 688\\
    8 & t^8 + 28t^6 + 112t^5 + 560t^4 + 1792t^3 + 4032t^2 + 5504t + 3456\\
    9 & t^9 + 36t^7 + 168t^6 + 1008t^5 + 4032t^4 + 12096t^3 + 24768t^2 + 31104t + 18032
    \end{array}
    $$
    \caption{The polynomials $\bh(\I_n^{4213},t)=\bh(\I_n^{4132},t)$ for small values of~$n$.}
    \label{tab:h4213}
\end{table}

For patterns $\sigma$ in class (ii), it was shown by Gessel~\cite{Gessel} that
$$|\S_n(\sigma)|=\frac{1}{(n+1)^2(n+2)}\sum_{k=0}^n\binom{2k}{k}\binom{n+1}{k+1}\binom{n+2}{k+1}.$$
Using Equation~\eqref{eq:h-Si} for the patterns $4321,4312,4123$, which are in class (ii), we deduce from Corollary~\ref{cor:refined_enumeration} that
\begin{align*}
\bh(\I_n^{4321},t)&=\bh(\I_n^{4312},t)=\bh(\I_n^{4123},t)\\
&=\sum_{i=0}^n  \binom{n}{i}\frac{1}{(i+1)^2(i+2)}\sum_{k=0}^i\binom{2k}{k}\binom{i+1}{k+1}\binom{i+2}{k+1}(t-1)^{n-i}.
\end{align*}

The polynomials $\bh(\I_n^{4321},t)=\bh(\I_n^{4312},t)=\bh(\I_n^{4123},t)$ for small values of $n$ are given in Table~\ref{tab:h4321}. The sequence of constant terms, listed in Table~\ref{tab:LSR}, does not currently appear in~\cite{OEIS}.

\begin{table}[htb]
    $$
    \begin{array}{r|c}
       n  & \bh(\I_n^{4321},t)=\bh(\I_n^{4312},t)=\bh(\I_n^{4123},t) \\ \hline
    1 & t \\ 2 & t^2+1 \\ 3 &  t^3+3t+2 \\ 4 & t^4 + 6t^2 + 8t + 8\\
    5 & t^5 + 10t^3 + 20t^2 + 40t + 32 \\
    6 & t^6 + 15t^4 + 40t^3 + 120t^2 + 192t + 145\\
    7 & t^7 + 21t^5 + 70t^4 + 280t^3 + 672t^2 + 1015t + 702\\
    8 & t^8 + 28t^6 + 112t^5 + 560t^4 + 1792t^3 + 4060t^2 + 5616t + 3598\\
    9 & t^9 + 36t^7 + 168t^6 + 1008t^5 + 4032t^4 + 12180t^3 + 25272t^2 + 32382t + 19280
    \end{array}
    $$
    \caption{The polynomials $\bh(\I_n^{4321},t)=\bh(\I_n^{4312},t)=\bh(\I_n^{4123},t)$ for small values of~$n$.}
    \label{tab:h4321}
\end{table}

Again, without relying on shellability of the corresponding complexes of pattern-avoiding injective words, there seems to be no simple combinatorial reason explaining why the number of ephemeral records has the same distribution on $\S_n(4321)$, $\S_n(4312)$ and $\S_n(4123)$. In particular, the proof of Wilf-equivalence of these patterns given by Backelin, West and Xin~\cite{BWX} does not appear to yield this refinement.

\begin{table}[htb]
$$\begin{array}{c|l}
\sigma & |\L_n(\sigma)| \text{ for } 1\le n\le 8\\ \hline
321,312 & 0, 1, 1, 3, 6, 15, 36, 91\\ \hline
4321,4312,4123 &  0, 1, 2, 8, 32, 145, 702, 3598\\ \hline
4231 &   0, 1, 2, 8, 32, 145, 703, 3616\\ \hline
4213,4132 &  0, 1, 2, 8, 32, 144, 688, 3456
\end{array}$$
\caption{Initial terms of the counting sequences of $\sigma$-avoiding permutations with no ephemeral records, for patterns $\sigma$ of length 3 and 4 starting with their largest entry.}
\label{tab:LSR}
\end{table}

\section{Separable permutations}\label{sec:separable}

Separable permutations are those that avoid both of the patterns $2413$ and $3142$. We denote the set of separable permutations of size $n$ by $\S_n(2413,3142)$, and we similarly generalize other notation such as $P_{2413,3142}(x)$ and $\I_n^{2413,3142}$, as well as the notion of $2413,3142$-blockers.

It is known~\cite{West} that separable permutations are enumerated by the large Schr\"oder numbers, and that the corresponding generating function is
\begin{equation}\label{eq:P_separable}
P_{2413,3142}(x)=\frac{3-x-\sqrt{1-6x+x^2}}{2}.
\end{equation}

The goal of this section is to prove that the complex $\I_n^{2413,3142}$ is shellable with respect to the lexicographic order. We will use a new shellability tool introduced in Section~\ref{sec:new_technique} as an ingredient in our proof.

\subsection{Characterization of blockers}\label{sec:characterizing_blockers}

We start by characterizing the $2413,3142$-blockers of any given $\pi\in\S_n(2413,3142)$ with respect to $\lex$. 
Throughout this section, we will use the term {\em blocker} to mean $2413,3142$-blocker with respect to $\lex$.

As shown in~\cite{BBL}, separable permutations can be described recursively as those that are obtained from smaller separable permutations by using {\em direct sum} and {\em skew sum} operations. For $\pi\in\S_n$ and $\tau\in\S_m$, their direct sum $\pi\oplus\tau$ and their skew sum $\pi\ominus\tau$ are permutations in $\S_{n+m}$ defined as
\begin{align*}    
(\pi\oplus\tau)_i&=\begin{cases}\pi_i & \text{if }1\le i\le n, \\ \tau_{i-n}+n & \text{if }n+1\le i\le n+m,\end{cases}\\
(\pi\ominus\tau)_i&=\begin{cases}\pi_i+m & \text{if }1\le i\le n, \\ \tau_{i-n} & \text{if }n+1\le i\le n+m.\end{cases}
\end{align*}

Following \cite{BBFGP}, we
we will represent separable permutations as {\em signed Schr\"oder trees with alternating signs}. These are rooted plane trees where every internal node represents either a $\oplus$ or a $\ominus$ operation on the permutations determined by its subtrees. In particular, each internal node has at least two children, and the leaves of $T(\pi)$ correspond to the entries of $\pi$. The children of a $\oplus$ node are either leaves or $\ominus$ nodes, and the children of a $\ominus$ node are either leaves or $\oplus$ nodes. See Figure~\ref{fig:separable} for an example. We use $\T_n$ to denote the set of signed Schr\"oder trees with alternating signs. It is shown in 
\cite[Proposition 2.13]{BBFGP} that this correspondence is a bijection between $\S_n(2413,3142)$ and $\T_n$.

\newcommand{\sq}[2]{\draw #1 rectangle ++(#2,#2);}

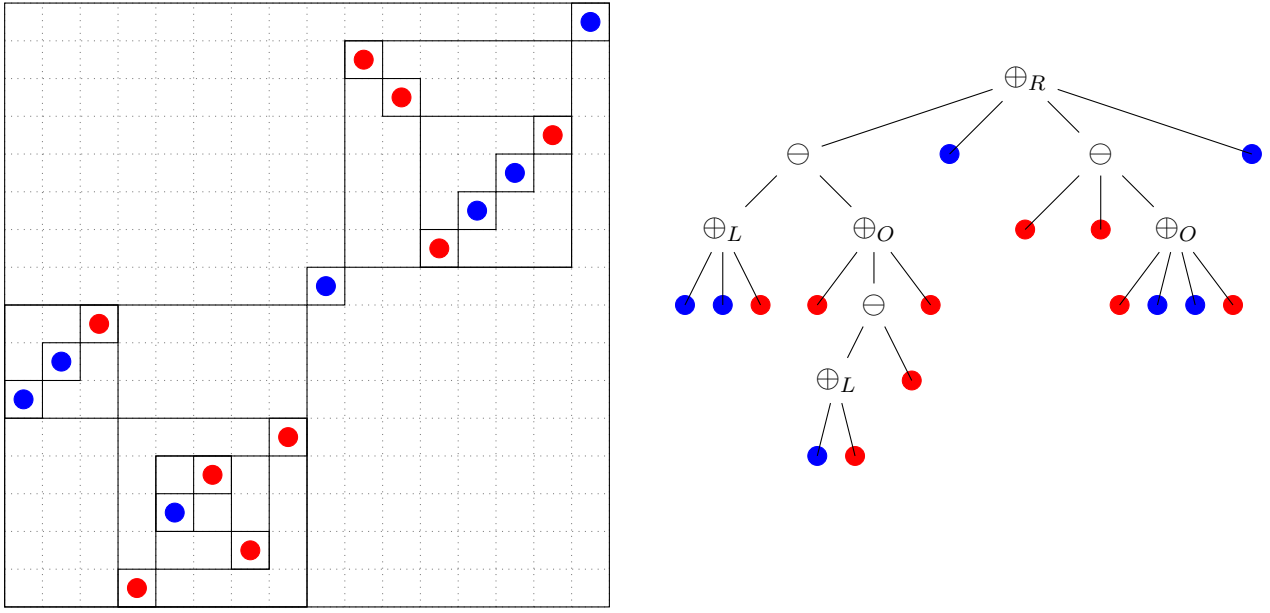
\begin{figure}[htb]
\begin{tikzpicture}[scale=.5]
\draw[dotted,gray] (0,0) grid (16,16);
\draw (0,0) rectangle (16,16);
\sq{(0,0)}{8}
\sq{(0,5)}{3}
\sq{(3,0)}{5}
\sq{(4,1)}{3}
\sq{(4,2)}{2}
\sq{(9,9)}{6}
\sq{(11,9)}{4}
\foreach \x in {(1,6),(2,7),(5,3),(9,9),(13,11),(14,12),(16,16)}{
    \filldraw[blue] \x+(-.5,-.5) circle (.25);
    \sq{\x}{-1}
}
\foreach \x in {(3,8),(4,1),(6,4),(7,2),(8,5),(10,15),(11,14),(12,10),(15,13)}{
    \filldraw[red] \x+(-.5,-.5) circle (.25);
    \sq{\x}{-1}
}
\begin{scope}[shift={(27,14)}]
\node (1) at (0,0) {$\oplus_R$};
\node (2) at (-6,-2) {$\ominus$};
\filldraw[blue] (-2,-2) circle (.25) coordinate (3);
\node (4) at (2,-2) {$\ominus$};
\filldraw[blue] (6,-2) circle (.25) coordinate (5);
\node (6) at (-8,-4) {$\oplus_L$};
\node (7) at (-4,-4) {$\oplus_O$};
\filldraw[red] (0,-4) circle (.25) coordinate (8);
\filldraw[red] (2,-4) circle (.25) coordinate (9);
\node (10) at (4,-4) {$\oplus_O$};
\filldraw[blue] (-9,-6) circle (.25) coordinate (11);
\filldraw[blue] (-8,-6) circle (.25) coordinate (12);
\filldraw[red] (-7,-6) circle (.25) coordinate (13);
\filldraw[red] (-5.5,-6) circle (.25) coordinate (14);
\node (15) at (-4,-6) {$\ominus$};
\filldraw[red] (-2.5,-6) circle (.25) coordinate (16);
\filldraw[red] (2.5,-6) circle (.25) coordinate (17);
\filldraw[blue] (3.5,-6) circle (.25) coordinate (18);
\filldraw[blue] (4.5,-6) circle (.25) coordinate (19);
\filldraw[red] (5.5,-6) circle (.25) coordinate (20);
\node (21) at (-5,-8) {$\oplus_L$};
\filldraw[red] (-3,-8) circle (.25) coordinate (22);
\filldraw[blue] (-5.5,-10) circle (.25) coordinate (23);
\filldraw[red] (-4.5,-10) circle (.25) coordinate (24);
\foreach \x in {2,...,5} {\draw (1) -- (\x);}
\draw (2) -- (6); \draw (2) -- (7);
\draw (4)--(8); \draw (4)--(9); \draw (4)--(10); 
\draw (6) -- (11); \draw (6) -- (12); \draw (6) -- (13);
\draw (7) -- (14); \draw (7) -- (15); \draw (7) -- (16);
\foreach \x in {17,...,20} {\draw (10) -- (\x);}
\draw (15)--(21); \draw (15)--(22);
\draw (21)--(23); \draw (21)--(24);
\end{scope}
\end{tikzpicture}
\caption{The separable permutation $\pi=6\,7\,8\,1\,3\,4\,2\,5\,9\,15\,14\,10\,11\,12\,13\,16$ and its tree $T(\pi)$. Blockers are colored in \textcolor{red}{red}, non-blockers are colored in \textcolor{blue}{blue}.}
\label{fig:separable}
\end{figure}

If $v$ is any node of $T(\pi)$, we will use $T_v$ to denote the subtree rooted at $v$.
The entries of $\pi$ corresponding to the leaves of $T_v$ form a {\em block} in $\pi$, meaning that their positions form an interval (i.e., they are consecutive) and so do their values. The {\em standardization} of this block is the permutation obtained by shifting down the values by the same amount, so that they become $\{1,2,\dots,r\}$ for some $r$.
We denote by $\st(T_v)$ the standardization of the block of entries corresponding to the leaves of $T_v$. For example, in Figure~\ref{fig:separable}, if $v$ is the $\ominus$ node of $T(\pi)$ that has three children, then the leaves of $T_v$ form the block $15\,14\,10\,11\,12\,13$, and $\st(T_v)=6\,5\,1\,2\,3\,4$.

In order to characterize the blockers in separable permutations, we will distinguish between three kinds of $\oplus$ nodes in $T(\pi)$. The root of the tree, if it is a $\oplus$ node, will be denoted by $\oplus_R$. Any $\oplus$ node that is the leftmost child of its parent (which is necessarily a $\ominus$ node) will be denoted by $\oplus_L$. Finally, any $\oplus$ node that is neither the root not the leftmost child of its parent will be denoted by $\oplus_O$ (where $O$ stands for ``other'').

\begin{lem}\label{lem:blockers_separable}
    Let $\pi\in\S_n(2413,3142)$. An entry $\pi_j$ is a $2413,3142$-blocker of $\pi$ with respect to $\lex$ if and only if the corresponding leaf of $T(\pi)$ satisfies any of the following conditions:
    \begin{enumerate}[(a)]
        \item it is a child of a $\ominus$ node,
        \item is it the leftmost or the rightmost child of a $\oplus_O$ node,
        \item it is the rightmost child of a $\oplus_L$ node,
    \end{enumerate}
\end{lem}

\begin{proof}
    Recall that an entry $\pi_j$ of $\pi\in\S_n(2413,3142)$ is a blocker if it can be moved within $\pi$ to create a separable permutation that is lexicographically smaller than $\pi$.

    Let $\ell$ denote the leaf of $T(\pi)$ corresponding to $\pi_j$. We start by showing that if $\ell$ satisfies any of the above conditions, then $\pi_j$ is a blocker.
    
    Suppose first that the parent of $\ell$ is a $\ominus$ node, which we call $p$ (for {\it parent}). 
    If $\ell$ is not the rightmost child of $p$, let $s$ (for {\it sibling}) be the child of $p$ immediately to the right of $\ell$. Then the leaves of the subtree $T_s$ correspond to a block $\pi_{j+1}\pi_{j+2}\dots\pi_r$ for some $j+1\le r\le n$, whose values form an interval with maximum $\pi_j-1$ (since p is a $\ominus$ node). Let $\pi'$ be the permutation obtained from $\pi$ by moving $\pi_j$ to position $r$, that is, shifting $\pi_j$ to the right past this block.
    Since $\pi_j>\pi_{j+1}$, we have $\pi'\lex\pi$. 
    To show that $\pi'$ is separable, let $\tau=\st(T_s)$, that is, the standardization of $\pi_{j+1}\pi_{j+2}\dots\pi_r$. In the recursive decomposition of $\pi$ as a separable permutation, the operation between the entry in position $j$ and the block of entries in positions $j+1,j+2,\dots,r$ is $1\ominus \tau$. Replacing this operation with $\tau\oplus 1$, we obtain the recursive decomposition of $\pi'$ as a separable permutation.

    If $\ell$ is the rightmost child of $p$, let now $s$ be the child of $p$ immediately to the left of $\ell$ (which must exist since $p$ has at least two children). Then the leaves of $T_s$ correspond to a block $\pi_{r}\pi_{r+1}\dots\pi_{j-1}$ for some $1\le r\le j-1$, whose values form an interval with minimum $\pi_j+1$. Let $\pi'$ be the permutation obtained from $\pi$ by moving $\pi_j$ to position $r$.
    Since $\pi_j<\pi_{r}$, we have $\pi'\lex\pi$. 
    Additionally, $\pi'$ is separable because, letting $\tau=\st(T_s)$, we can obtain the recursive decomposition of $\pi'$ as a separable permutation by changing the operation $\tau\ominus 1$ to $1\oplus\tau$ in the decomposition of $\pi$.

    Suppose now that $\ell$ is the leftmost child of a $\oplus_O$ node $p$. In particular, $\st(T_p)=1\oplus \rho$ for some permutation $\rho$.
    Since $p$ is not the leftmost child of its parent, there must be another child of the same parent that is immediately to the left of $p$, let us call it $u$ (for {\it uncle}). The leaves of $T_u$ correspond to a block $\pi_{r}\pi_{r+1}\dots\pi_{j-1}$ for some $1\le r\le j-1$. Let $\pi'$ be the permutation obtained from $\pi$ by moving $\pi_j$ to position $r$.
    Since $\pi_j<\pi_{r}$, we have $\pi'\lex\pi$. 
    Additionally, $\pi'$ is separable because, letting $\tau=\st(T_u)$, we can obtain the recursive decomposition of $\pi'$ as a separable permutation by changing the operations $\tau\ominus (1\oplus \rho)$ to $1\oplus(\tau\ominus\rho)$ in the decomposition of $\pi$.

    Next suppose that $\ell$ is the rightmost child of a $\oplus_O$ node $p$. In particular, $\st(T_p)=\rho \oplus 1$ for some permutation $\rho$. Let $u$ be the sibling of $p$ immediately to its left, and let $\tau=\st(T_u)$.
    Let $\pi'$ be the permutation obtained from $\pi$ by moving $\pi_j$ to the left past all the entries corresponding to the leaves of $T_u$ and $T_p$ (other than $\ell$ itself). Since the entries corresponding to the leaves $T_u$ are all larger than $\pi_j$, we have $\pi'\lex\pi$.
    Additionally, $\pi'$ is separable because we can obtain the recursive decomposition of $\pi'$ as a separable permutation by changing the operations $\tau\ominus (\rho \oplus 1)$ to $(1\oplus\tau)\ominus\rho$ in the decomposition of $\pi$.

    Finally, suppose that $\ell$ is the rightmost child of a $\oplus_L$ node $p$. Again, $\st(T_p)=\rho \oplus 1$ for some permutation $\rho$. Now let $u$ be the sibling of $p$ immediately to its right, and let $\tau=\st(T_u)$.
    Let $\pi'$ be the permutation obtained from $\pi$ by moving $\pi_j$ to the right past all the entries corresponding to the leaves of $T_u$. Since these entries are smaller than  $\pi_j$, we have $\pi'\lex\pi$.
    Additionally, $\pi'$ is separable because we can obtain the recursive decomposition of $\pi'$ as a separable permutation by changing the operations $(\rho \oplus 1)\ominus\tau$ to $(\rho\ominus\tau)\oplus 1$ in the decomposition of $\pi$.
    
    It remains to prove that if $\ell$ does not satisfy any of the conditions (a), (b), (c), then $\pi_j$ is not a blocker.
    A leaf that does not satisfy these conditions must be a {\em middle} (that is, neither leftmost not rightmost) child of any kind of $\oplus$ node, a leftmost child of a $\oplus_L$ node, or a leftmost or rightmost child of a $\oplus_R$ node. 

    Suppose first that $\ell$ is a middle child of some $\oplus$ node $p$. In particular, $\pi_{j-1}<\pi_j<\pi_{j+1}$. Since $\pi_j<\pi_{j+1}$, the only way for $\pi_j$ to be a blocker is if there exists some $i<j$ such that $\pi_i>\pi_j$, and moving $\pi_j$ to position $i$ produces a separable permutation $\pi'$. Since $p$ is a $\oplus$ node, the conditions $i<j$ and $\pi_i>\pi_j$ imply that the leaf corresponding to $\pi_i$ is not in $T_p$. And since the values of the entries corresponding to leaves in $T_p$ form an interval, it follows that $\pi_i>\pi_{j+1}$ as well. But then $\pi_j\pi_i\pi_{j-1}\pi_{j+1}$ would be an occurrence of $2413$ in $\pi'$, contradicting that it is separable.

    Suppose now that $\ell$ is the leftmost child of a $\oplus_L$ node $p$. In particular, $\pi_j<\pi_{j+1}$. Let $g$ be the parent of $p$, and let $u$ be the sibling of $p$ immediately to its right. Let $a$ be an entry corresponding to a leaf of $T_u$, and note that $\pi_j>a$, since $g$ is a $\ominus$ node.
    Since $\pi_j<\pi_{j+1}$, the only way for $\pi_j$ to be a blocker is if there exists some $i<j$ such that $\pi_i>\pi_j$, and moving $\pi_j$ to position $i$ produces a separable permutation $\pi'$.
    For such $i$ to exist there must be some ancestor $h$ of $\ell$ that is a $\oplus_O$ node, so that a left sibling of this ancestor contains the leaf corresponding to $\pi_i$. Note that $\pi_i$ is larger than the entries corresponding to all the leaves of $T_h$.
    Since $g$ is not the root (because $h$ is an ancestor), it must have some sibling $g'$. Let $b$ be an entry corresponding to a leaf of $T_{g'}$.
    
    If $g'$ is to the right of $g$, then $b$ is larger than all the entries corresponding to leaves of $T_g$. In this case, $\pi'$ contains the subsequence $\pi_j\pi_iab$, which is an occurrence of $2413$.
    On the other hand, if $g'$ is to the left of $g$, then $b$ is smaller than all the entries corresponding to leaves of $T_g$. In this case, $\pi'$ contains the subsequence $\pi_j\pi_ib\pi_{j+1}$, which is again an occurrence of $2413$, contradicting the fact that $\pi'$ is separable.

    Finally, if $\ell$ is the leftmost child of the root $\oplus_R$, then $j=1$ and $\pi_j=1$, which clearly cannot be moved to produce a lexicographically smaller permutation. Similarly, if $\ell$ is the rightmost child of the root $\oplus_R$, then $j=n$ and $\pi_j=n$, which is not a blocker either.
\end{proof}

\subsection{Refined enumeration by blockers}

Next we use the characterization from Lemma~\ref{lem:blockers_separable} to enumerate separable permutations with respect to the number of blockers. 
Let $S(t,x)$ be the ordinary generating function for separable permutations where $x$ marks the size and $t$ marks the number of blockers. Via the above bijection between $\S_n(2413,3142)$ and $\T_n$, this is also the generating function for the set $\T=\bigcup_{n\ge0}\T_n$ where $x$ marks the number of leaves, and $t$ marks the number of leaves satisfying the conditions in Lemma~\ref{lem:blockers_separable}.

\begin{prop}\label{prop:S}
    The generating function for separable permutations with respect to size and the number of blockers is
    $$S(t,x)=\frac{3-(3-2t)x-\sqrt{1-2(1+2t)x+(1+4t-4t^2)x^2}}{2\left(1-(1-t)x\right)^2}.$$
\end{prop}

\begin{proof}
We will obtain an equation for $S\coloneqq S(t,x)$ by using a recursive decomposition of trees in $\T$. Define generating functions $S^\ominus\coloneqq S^\ominus(t,x)$, $S^{\oplus_R}\coloneqq S^{\oplus_R}(t,x)$, $S^{\oplus_L}\coloneqq S^{\oplus_L}(t,x)$, $S^{\oplus_O}\coloneqq S^{\oplus_O}(t,x)$ for subtrees where the root is a $\ominus$, $\oplus_R$, $\oplus_L$ and $\oplus_O$ node, respectively, again with $x$ marking the number of leaves and $t$ marking the number of leaves meeting the conditions in Lemma~\ref{lem:blockers_separable}. These generating functions satisfy the following system of equations:
\begin{align}
    S & = 1+x+S^{\oplus_R}+S^\ominus, \label{eq:S}\\
    S^{\oplus_R} & = \frac{(x+S^\ominus)^2}{1-x-S^\ominus}, \label{eq:SR}\\
    S^{\oplus_L} & = \frac{(x+S^\ominus)(tx+S^\ominus)}{1-x-S^\ominus}, \label{eq:SL}\\
    S^{\oplus_O} & = \frac{(tx+S^\ominus)^2}{1-x-S^\ominus}, \label{eq:SO}\\
    S^{\ominus} & = \frac{(tx+S^{\oplus_L})(tx+S^{\oplus_O})}{1-tx-S^{\oplus_O}}. \label{eq:Sminus}
\end{align}

Indeed, Equation~\eqref{eq:S} holds because a separable permutation is either empty, has size $1$ (contributing $x$ since it has no blockers), or has size at least $2$, in which case it is represented by a tree in $\T$ whose root is either a $\oplus_R$ or a $\ominus$ node.
Equation~\eqref{eq:SR} describes the fact that the sequence of subtrees of a $\oplus_R$ has at least two elements, and each of them is either a leaf (contributing $x$ since it is not a blocker by Lemma~\ref{lem:blockers_separable}) or a tree rooted at a $\ominus$ node (contributing $S^\ominus)$. 
For Equation~\eqref{eq:SL}, again we use that the subtrees of a $\oplus_L$ are either leaves or trees rooted at a $\ominus$ node, but the difference here is that if the rightmost child of a $\oplus_L$ node is a leaf, then it is a blocker by Lemma~\ref{lem:blockers_separable}(c), so it contributes $tx$.
For Equation~\eqref{eq:SO}, we use that both the leftmost and the rightmost children of a $\oplus_O$ node, if they are leaves, are blockers by Lemma~\ref{lem:blockers_separable}(a), so they contribute $tx$.
Finally, to obtain Equation~\eqref{eq:Sminus}, we use that the subtrees of a $\ominus$ are either leaves (which are blockers by Lemma~\ref{lem:blockers_separable}(a)) or trees rooted at a $\oplus_L$ node (the leftmost child) or a $\oplus_O$ node (the other children).

To solve this system of equations, first substitute the expressions for $S^{\oplus_L}$ and $S^{\oplus_O}$ from Equations~\eqref{eq:SL} and~\eqref{eq:SO} into Equation~\eqref{eq:Sminus} to get a quadratic equation involving only $S^{\ominus}$. Solving for $S^{\ominus}$ gives
$$S^\ominus=\frac{1-(1+2t)x-\sqrt{1-2(1+2t)x+(1+4t-4t^2)x^2}}{4}.$$
Substituting this expression into Equation~\eqref{eq:SR} we get a formula for $S^{\oplus_R}$, which we can then use in Equation~\eqref{eq:S} to obtain the stated expression for $S(t,x)$.
\end{proof}

\subsection{Proving shellability}
We can now use the tools from Section~\ref{sec:new_technique} to prove shellability for the complex of separable injective words.

\begin{thm}
     The complex $\I_n^{2413,3142}$ is pure of dimension $n-1$ and is shellable with respect to $\lex$.
\end{thm}

\begin{proof}
It is proven in Corollary~\ref{cor:separable-pure} that the complex $\I_n^{2413,3142}$ is pure of dimension $n-1$.
Using Proposition~\ref{prop:GFh} and Equation~\eqref{eq:P_separable}, we obtain the generating function for the $h$-polynomials of this complex:
\begin{align} \nonumber
    \sum_{n=0}^\infty h(\I_n^{2413,3142},t)\,x^n
&=\frac{1}{1-(1-t)x}\,P_{2413,3142}\left(\frac{tx}{1-(1-t)x}\right)\\
&=\frac{3-(3-2t)x-\sqrt{1-2(1+2t)x+(1+4t-4t^2)x^2}}{2\left(1-(1-t)x\right)^2}. \label{eq:separable_hvector}
\end{align}

Comparing Equation~\eqref{eq:separable_hvector} with Proposition~\ref{prop:S}, we see that
$$\sum_{n=0}^\infty h(\I_n^{2413,3142},t)\,x^n=S(t,x),$$
so the $h$-polynomial of $\I_n^{2413,3142}$ enumerates permutations in $\S_n(2413,3142)$ according to the number of  $2413,3142$-blockers with respect to $\lex$. Using Lemma~\ref{lem:hvector_implies_shelling}, we deduce that the lexicographic order on $\S_n(2413,3142)$ is a shelling order of the facets of $\I_n^{2413,3142}$. 
\end{proof}

With the above setup, we can also easily compute the top $h$-numbers $h_n(\I_n^{2413,3142})$.
Combining Equations~\eqref{eq:Euler_characteristic} and~\eqref{eq:P_separable}, we deduce that 
\[
    \sum_{n=0}^\infty h_n(\I_n^{2413,3142})\,x^n =
\frac{3+2x-\sqrt{1-4x-4x^2}}{2(1+x)^2}.
\]
The coefficients appear as sequence \cite[A052705]{OEIS}, which also arises in lattice path enumeration.

\section{Homology bases}\label{homology-section}

In this section, we construct a homology basis for $\I_n^{\sigma }$ for each each extremal-endpoint pattern $\sigma $.  We do this first for those patterns $\sigma $ whose largest letter is leftmost,
using as a key ingredient the homology facets determined by the shellings constructed earlier in the paper. 
We then use symmetries (which yield isomorphisms of complexes that serve as transfer maps), allowing us to construct 
homology bases for the complexes  $\I_n^{\sigma }$ for any pattern $\sigma $ 
whose largest or smallest letter is leftmost or rightmost, i.e., any extremal-endpoint pattern.  In particular, this gives homology bases for all patterns of length 3 as well as for most patterns of length 4.

\begin{remark} The homology basis construction in \cite{MW} for the full complex of injective words  bears some resemblance to the homology bases we will provide now for the subcomplexes $\I_n^{\sigma }$ of the complex of injective words that avoid a pattern  $\sigma $ 
that starts with its largest letter.
However, there are significant additional subtleties  to be overcome  in turning to subcomplexes avoiding a pattern.
\end{remark}

As mentioned above, each of our shellings in the previous sections specifies  an associated  collection of homology facets, namely the collection of facets which attach along their entire boundaries in the shelling steps which insert these facets. It is well known (and relatively easy to prove e.g.\ using discrete Morse theory)  
that  the number of homology facets in a shelling of a Boolean cell complex $\Delta $ equals the rank of the top homology group of $\Delta $.

In what follows, we describe and justify  for a pattern $\sigma $ whose largest letter is leftmost  how to use each homology facet $\pi $ in the lexicographic shelling order for $\I_n^{\sigma }$ to specify an alternating sum 
$z_{\pi }$ of facets that we prove  is a homology class.  Moreover, we  do so in such a way that the collection of homology classes  obtained this way from the various homology facets comprises a homology basis.

Before turning to this construction, we recall how the boundary map  is defined for $\I_n$ and its subcomplexes.

 \begin{defn}
     Define the boundary map $d_r: C_r\rightarrow C_{r-1}$ from the $r$-th chain group to the $(r-1)$-st chain group for $\I_n$
     as the linear 
     map  
     satisfying 
    $$d_r(a_1\cdots a_{r+1}) = \sum_{i=0}^{r+1}(-1)^{i-1}a_1\cdots \hat{a}_i\cdots a_{r+1}$$ for any injective word $a_1\cdots a_{r+1}$ of length $r+1$.
 \end{defn}

Recall from Lemma~\ref{thm:shellability_start_with_k}
that, for the full complex of injective words, a permutation $\pi $ gives rise to  a homology facet in the shelling order which orders permutations  (written in one-line notation) lexicographically  if and only if each letter $\pi_i$  in $\pi $ that is strictly larger than all letters to its left is immediately followed by a letter  $\pi_{i+1}$ that is strictly smaller than it.   We have shown in Corollary ~\ref{homology-facet-description} that this same characterization of homology facets holds for $\I_n^{\sigma }$ for  any permutation pattern $\sigma $ whose largest letter is leftmost.

Next are a few simple examples of homology facets $\pi $ in $\I_n^{321}$ and their associated homology classes $z_{\pi }$.  Later in this section, we will describe and justify the general construction which yields these homology classes.

\begin{ex}\label{simplest-examples}
The homology facet $\pi = 2143$ in $\I_4^{321}$ gives rise to a homology class $[z_{\pi }]$ for   $$z_{\pi } = 2143 + 1243 + 2134 + 1234.$$
The  homology facet $\pi = 3142$ in $\I_4^{321}$ likewise yields the  homology class  $[z_{\pi }]$ given by 
$$z_{\pi } = 3142 + 1342 + 3124 + 1324.$$ 

We leave it to the interested reader to apply the boundary map to  $z_{\pi }$ in either case  above and check that the boundary is indeed 0.
\end{ex}

In Example~\ref{simplest-examples}, all of the letters in the homology facets are either left-to-right maxima or their immediate succesors, making these examples considerably simpler than is the case for arbitrary homology facets.  The next example is more indicative of the general case.

Before turning to this example,  we first introduce a nonstandard way to decompose the boundary map that will play a crucial role both in understanding this example and more generally in justifying that our proposed cycles are indeed cycles.

 \begin{defn}\label{def:boundary-splitting}
Let $d_{\underline{i}}$ denote  the part of the boundary map $d_r$ (for any $r\le n-1$) which deletes from any injective word the letter $i$ from whatever position $j$ it occurs in with coefficient $(-1)^{j-1}$.  If $i$ does not appear in $a_1\cdots a_{r+1}$, then 
 $d_{\underline{i}}(a_1\cdots a_{r+1})=0$.  Thus, the boundary map $d_r$ for $\I_n$ satisfies $d_r = \sum_{i=1}^n d_{\underline{i}}$.    As a simple example, 
 $d_{\underline{2}}(31425) = (-1)^{3}3145$ since $2$ appears in the 4-th position in 31425, whereas $d_{\underline{2}}(3145)=0$ since 2 does not appear at all in 3145.
\end{defn}

\begin{ex}\label{slightly-more-involved-example}
 Given the homology facet $\pi = 312$  in $\I_3^{321}$, 
 we obtain a  homology class $[z_{\pi }]$ by setting $z_{\pi } :=  312 + 132 - 213 - 231$.   

 One may check that $[z_{\pi }] $ a homology class by showing that $d_{\underline{i}}(z_{\pi })=0$ for each $i$.  This is done by  using the equations $d_{\underline{i}} (312) = -d_{\underline{i}}(132)$
 for $i=1,3$ and  $d_{\underline{i}}(-213) = -d_{\underline{i}}(-231)$ for $i=1,3$ to show that
 $d_{\underline{i}}(z_{\pi })=0$ 
 for $i=1,3$.  Then one may use 
 $d_{\underline{2}}(312) = -d_{\underline{2}} (-231) $ and $d_{\underline{2}}(132) = -d_{\underline{2}}(-213)$ to show that $d_{\underline{2}}(z_{\pi })=0$.
 Since $d(z_{\pi }) = \sum_i d_{\underline{i}}(z_{\pi }) = 0$, it follows that $z_{\pi }$ is a cycle, namely has boundary 0.
 \end{ex}

\subsection{Patterns whose largest letter is leftmost}

Let $\pi $ be any homology facet in $\I_n^{\sigma }$ for a pattern  $\sigma $ 
whose largest letter is leftmost.  Let $L(\pi )$ denote the set of positions of the left-to-right maxima in $\pi $.  Let $U(\pi ) = \{ j\in [n]| j\not\in L(\pi )\text{ and } j-1\not\in L(\pi )\} $.  In other words, $U(\pi )$ is the set of positions of letters that are neither left-to-right  
maxima nor immediate successors of left-to-right maxima.  

We make the convention in our examples of underlining each left-to-right maximum as well as its immediate successor in any homology facet $\pi $;  that is, we underline the letters in $\pi $ whose positions are not in $U(\pi )$.  For each value that is underlined in $\pi $, we also underline this same value in each permutation appearing in $z_{\pi }$.  For example, we underline the values  3 and 1 in each permutation appearing in $z_{312}$, hence we write $z_{312} = \underline{31}2 + \underline{13}2 - 2\underline{13} - 2\underline{31}$.

\begin{ex}
Consider the permutation $\pi = \underline{61}2\underline{73}4\underline{85}$.  In this case, the left-to-right maxima are the values $6,7,8$ which appear in positions $1,4,7$.  Thus, $L(\pi ) = \{ 1,4,7\}.$ 
In this case, $U(\pi ) = \{ 3,6 \} $.
\end{ex}

For each $\pi $ that is a homology facet in $\I_n^{\sigma }$ and each $T\subseteq U(\pi )$, denote by $\shift(T)$ the operator  which permutes the letters in $\pi $ 
as follows. The operator $\shift(T)$ proceeds  from left to right through the letters of $\pi $ whose indices are in $T$, shifting each such letter in turn leftward to the rightmost position which makes it a left-to-right maximum.  When we speak of the sign of $\shift(T)$, we are speaking of the sign of this permutation on positions of the letters in $\pi $.

\begin{ex}
Applying $\shift(T)$ to the permutation $\underline{41}2\underline{73}6\underline{85}$ for $T=\{ 3\} $ yields $$\shift(\{ 3\} )\pi = 2\underline{41}\underline{73}6\underline{85}.$$  Likewise, 
$\shift(\{ 6\} )\pi = \underline{41}26\underline{73}\underline{85}$ and $\shift(\{ 3,6\} (\pi ) = 2\underline{41}6\underline{73}\underline{85}$.

Applying $\shift(T)$ to the permutation 
$\underline{61}2\underline{73}4\underline{85}$ for $T=\{ 3,6 \}$  yields $24\underline{61}\underline{73}\underline{85}$.

As a more complicated example, consider the permutation 
$\pi = \underline{64}5\underline{98}1237$.  Then applying $\shift(\{ 6,7,8,9\} )$ to $\pi $ yields $123\underline{64}57\underline{98}$ whereas applying $\shift(\{ 7,8,9\})$ to $\pi $ yields $23\underline{64}57\underline{98}1$ and applying $\shift (\{ 6,7,9\} )$ to $\pi $  yields $12\underline{64}57\underline{98}3$.
\end{ex}

Next we state and justify three important  properties of $\shift(T)$. These properties  will be used in justifying that the permutations appearing in the proposed cycles  $z_{\pi }$ that we introduce shortly  are all $\sigma $-avoiding and that $d(z_{\pi })=0$. These will also help explain why the operator $\prod_{m\in M}(\pi_m,\pi_{m+1})$ introduced shortly  is indeed a well-defined and well-behaved operator.  

\begin{obs}\label{stay-maxima}
Let $\pi $ be a homology facet in $\I_n^{\sigma }$.
If $\pi_m$ is a left-to-right maximum in $\pi $, then $\pi_m$ is also a left-to-right maximum in $\shift(T)[\pi ]$ for each $T\subseteq U(\pi )$. 
\end{obs}
\begin{proof}
Simply note that a letter $\pi_j\in U(\pi )$ will never be shifted leftward by the operator $\shift(T)$ past a left-to-right maximum that is smaller than it, due to each letter being shifted leftward the minimal amount necessary to make it a left-to-right maximum.
\end{proof}

\begin{obs}\label{stay-successor}
Let $\pi $ be a homology facet in $\I_n^{\sigma }$.  For each left-to-right maximum $\pi_m\in L(\pi )$, its immediate successor $\pi_{m+1}$ remains its immediate successor in $\shift(T)[\pi ]$ for each $T\subseteq U(\pi )$.
\end{obs}
\begin{proof}
Note that any letter $\pi_j\in U(\pi )$ that would be shifted leftward past $\pi_{m+1}$ but not past $\pi_m$ would necessarily satisfy $\pi_m < \pi_j < \pi_{m+1}$.  But having $\pi_m < \pi_{m+1}$ for $\pi_m\in L(\pi )$ contradicts the property of a homology facet $\pi $ that it  must  have  $\pi_m > \pi_{m+1}$ for each left-to-right maximum  $\pi_m$ in $ \pi $. 
\end{proof}

The next property of our shifting map will be used to prove that our upcoming proposed cycles $[z_{\pi }]$ introduced in Definition~\ref{def:homology-cycle} are indeed cycles (i.e. have 0 boundary).

\begin{obs}\label{toughest-observation-to-prove}
If $S\subseteq U(\pi )$ satisfies $S = T \cup \{ i\} $ for some $i\not\in T$,
then $\shift(S)[\pi ]$ may be obtained from $\shift(T)[\pi ]$ by shifting the letter $\pi_{i}$ leftward from its position in $\shift(T)[\pi ]$ by the minimal amount necessary to make it a left-to-right maximum. 
\end{obs}

\begin{proof}
This is proven in Lemma~\ref{stay-consistent}.
\end{proof}

For each $M\subseteq L(\pi )$, denote by $\prod_{m\in M}(\pi_m,\pi_{m+1})$ the permutation which applies  the adjacent transpositions $(\pi_m,\pi_{m+1})$ for $m\in M$ swapping the values $\pi_m$ and $\pi_{m+1}$ where $\pi_m$ is one of the left-to-right maxima in $\pi $  and $\pi_{m+1}$ is its immediate successor in $\pi $.  By Observation~\ref{stay-successor}, $\pi_{m+1}$ is the immediate successor of $\pi_m$ within $\shift(T)[\pi ]$ for each $m\in M$, ensuring that this operator $\prod_{m\in M}(\pi_m,\pi_{m+1})$ acts as  a product of adjacent transpositions on positions which all commute with each other.

\begin{ex}
Consider the homology facet  $\pi = \underline{31}2$.
The upcoming formula for $z_{\pi }$ in Definition~\ref{def:homology-cycle} will yield the
the cycle
 $z_{312} = \underline{31}2 + \underline{13}2 - 2\underline{31} - 2\underline{13}$.  This is obtained by  the following two step process.  First we replace $\underline{31}2$ with $\underline{31}2 +(-1)^{1+2} 2\underline{31}$, thereby taking a signed sum over choices of $T\subseteq U(\pi )$, applying the permutation 
 $\shift(T)$ to $\pi $ for each choice of $T$ and assigning the  coefficient $(-1)^{|T|}\sgn (\shift(T))$ to this summand.
 In the second step, we  replace $\underline{31}2$ by 
 $\underline{31}2+\underline{13}2$, and we replace $-2\underline{31}$ by $-(2\underline{31}+2\underline{13})$, thereby applying the operator $\prod_{m\in M}(\pi_m,\pi_{m+1})$ for each choice of $M\subseteq L(\pi )$ to the signed sum resulting from the first step.
 \end{ex}

Now  we formally define  the elements of our homology basis for the top homology of $\I_n^{\sigma }$ for any pattern $\sigma $ whose largest letter is leftmost.

\begin{defn}\label{def:homology-cycle}
Let $\sigma$ be a pattern that starts with its largest letter.  
We define a  homology class $[z_{\pi }]$ associated  to any  homology facet $\pi $ in $\I_n^{\sigma }$ as follows.  Set
$$z_{\pi } = \sum_{M\subset L(\pi )}(\prod_{m\in M}(\pi_m,\pi_{m+1}))
\sum_{T\subset U(\pi )} ((-1)^{|T|}\cdot \sgn(\shift(T)))(\shift(T)[\pi ] ) ,$$ letting $\sgn$ denote the usual  sign function on permutations.   
\end{defn}

Next is a more involved example, meant to give some intuition both for the construction and for the fact (to be justified shortly) that applying the boundary map to any such $z_{\pi }$ indeed yields 0.

 \begin{ex} 
The $321$-avoiding permutation $\underline{61}2\underline{73}4\underline{85}$ includes the underlined pairs of letters  $(6,1), (7,3)$ and $(8,5)$ consisting of the left-to-right maxima and their immediate successors.  Note that this permutation $\pi = 61273485$ is a homology facet in $\I_8^{321}$.  Let us show how to construct the homology class $[z_{\pi }]$ associated to it.   For this permutation $\pi $, we  have $U(\pi ) = \{ 3,6\} $ 
and $L(\pi ) = \{ 1,4,7\}$.

First we calculate the inner alternating  sum for $M= \emptyset $ in the formula for $z_{\pi }$.  To that end, we  take an  alternating sum of permutations obtained from $\pi $ by applying the map
$\shift(T)$ for the various choices of set $T\subseteq U(\pi )$ of unmatched letters.  This yields:
$$\underline{61}2\underline{73}4\underline{85} - (-1)^2 2\underline{61}\underline{73}4\underline{85} - (-1)^5 4\underline{61}2\underline{73}\underline{85} + (-1)^2(-1)^4 24\underline{61}\underline{73}\underline{85}$$
with summands given by $T=\emptyset $, $T=\{ 3\} $, $T= \{ 6\} $ and $T=\{ 3,6\} $, respectively.  

Next we turn to the outer sum.  To obtain the full expression, we  replace each summand from step one  by a sum over all permutations obtained from it by
choosing each possible  subset $M$  of the original set of left-to-right maxima and swapping each of those belonging  to $M$  with its  immediate successor.  For example, this replaces 
$\underline{61}2\underline{73}4\underline{85}$
with the sum
$$ \underline{61}2\underline{73}4\underline{85} + \underline{16}2\underline{73}4\underline{85} + \underline{61}2\underline{37}4\underline{85} + \underline{16}2\underline{37}4\underline{85} + \underline{61}2\underline{73}4\underline{58} + \underline{16}2\underline{73}4\underline{58} + \underline{61}2\underline{37}4\underline{58} + \underline{16}2\underline{37}4\underline{58},$$ given by the $2^3$ choices of sets $M$.  
Likewise
$-2\underline{6173}4\underline{85}$ is replaced  with the sum
$$-(2\underline{61}\underline{73}4\underline{85} + 2\underline{16}\underline{73}4\underline{85} + 2\underline{61}\underline{37}4\underline{85} + 2\underline{16}\underline{37}4\underline{85} + 2\underline{61}\underline{73}4\underline{58} + 2\underline{16}\underline{73}4\underline{58} + 2\underline{61}\underline{37}4\underline{58} + 2\underline{16}\underline{37}4\underline{58}  ).$$  Altogether there are $4\cdot 8 = 32$ signed summands in $z_{\pi }$.

One may  check for $i=1,2,\dots ,8$  
that the component $d_{\underline{i}}$ of the boundary deleting the letter $i$ from each summand in $z_{\pi }$ with appropriate sign  
applied to $z_{\pi }$  yields 0.
Checking this for $i=1,2,\dots ,8$  shows that $z_{\pi }$ is a cycle.  The fact that these partial boundaries are 0 comes from the pairwise cancellation of terms in the boundary, as justified in Lemmas~\ref{cancel-max-and-successor} and~\ref{cancel-non-max}.

As one example of a pair of summands cancelling, notice that
$$d_{\underline{2}}(\underline{16}2\underline{73}4\underline{85} - 2\underline{16}\underline{73}4\underline{85}) = (-1)^2(1673485 - (-1)^0(1673485) = 0.$$
Another example, again using the same term $\underline{16}2\underline{73}4\underline{85}$, but now paired with a different cancellation  partner is
$$d_{\underline{4}}(\underline{16}2\underline{73}4\underline{85} + 4\underline{16}2\underline{73}\underline{85}) = (-1)^5 1627385 + (-1)^0 1627385 = 0.$$  This pairing of summands to be cancelled depends on the value $i$ in $d_{\underline{i}}$ and  is explained and justified  in Lemmas~\ref{stay-consistent} and~\ref{cancel-non-max} below.
\end{ex}

First we verify that all of the permutations appearing as summands in our formula for $z_{\pi }$ in Definition~\ref{def:homology-cycle} are $\sigma $-avoiding permutations.  This is accomplished in the next two lemmas.  

\begin{lem}\label{well-defined-shifted-permutation}
If $\pi $ is $\sigma $-avoiding, then $\shift(T)[\pi ]$ is also 
$\sigma $-avoiding for each $T\subseteq U(\pi )$.
\end{lem}

\begin{proof}
First we show that this leftward shifting  process carried out by $\shift(T)$ 
preserves the properties (1)  that each of the  original left-to-right maxima $\pi_m$ remains a left-to-right maxima, and (2) that each such $\pi_m$ still has the same immediate successor $\pi_{m+1}$.  This all follows from the fact that each of the letters  $t\in T$ gets shifted to the left of all of the original left-to-right maxima to its left which are larger than it, but will not get shifted to the left of any of the original  left-to-right maxima $\pi_{\max }$ to its left that are smaller than it, and additionally  will not get shifted to the left of the immediate successor of any $\pi_{\max } $ remaining  to its left (by virtue of $t$ getting shifted to the rightmost position with the property that $t$  becomes a left-to-right maximum, which implies once it is  in this position that its immediate successor will be larger than it and will itself be a left-to-right maxima.

Next  we use  properties (1) and (2) to deduce that the end result of this shifting  process is a $\sigma $-avoiding permutation.  
If the resulting permutation  $\pi '\coloneqq  \shift(T)(\pi )$ did include  a pattern $\sigma $, the occurrence of this pattern would necessarily have to  use one or more  of the shifted letters.  Since these shifted letters  are all left-to-right maxima after shifting, they may only appear as  the leftmost letter in the pattern $\sigma $.  Suppose there were such a pattern $\sigma $ appearing in $\pi '$.  Consider the rightmost copy of the forbidden pattern $\sigma $ within $\pi '$, more specifically, consider one with its leftmost letter as far right as possible.  By construction, 
each of the letters in $T$, namely each of the letters shifted leftward to obtain $\pi '$ from $\pi $,  is immediately succeeded by another left-to-right maximum in $\pi '$.  But  the leftmost letter in the pattern must be in $T$, since the pattern was not present in $\pi $.  This leftmost letter has a still larger letter to its immediate right which may replace it as the leftmost letter in another copy of the pattern $\sigma $ within $\pi '$.  This contradicts our choice of rightmost pattern within $\pi '$.  
\end{proof}

Having now shown that all of the permutations obtained in step 1, namely the leftward shifting process, are $\sigma $-avoiding, we next show that step 2 also preserves the property of being $\sigma $-avoiding.

\begin{lem}
If $\shift(T)[\pi ]$ is $\sigma $-avoiding  for a given $T\subseteq U(\pi )$, then 
$$(\prod_{m\in M}(\pi_m,\pi_{m+1}))(\shift(T)[\pi])$$ is also 
$\sigma $-avoiding for each $M\subseteq L(\pi )$.
\end{lem}

\begin{proof}
Otherwise some letter $\pi_m$ along with its successor $\pi_{m+1}$ must both appear in an appearance of the pattern $\sigma $ once $\pi_{m+1}$ is moved to the left of $\pi_m$.  But $\pi_{m}$ is still a left-to-right maximum even after it is moved to the right of $\pi_{m+1}$, which means it may only appear as the largest (and leftmost) letter in the pattern $\sigma $.  This contradicts $\pi_{m+1}$, which now appears to its left,  also appearing  in the same  pattern $\sigma $. 
\end{proof}

Combining  these two lemmas yields  the following:

\begin{prop}
For $\pi $ any homology facet of $\I_n^{\sigma }$, each permutation in the support of $z_{\pi }$ is 
$\sigma $-avoiding.
\end{prop}

Next we prove that each such  $z_{\pi }$ has boundary 0.  To this end, we decompose the boundary as $d=\sum_{i=1}^n d_{\underline{i}}$ where the operator $d_{\underline{i}}$  applied to a permutation $\tau $ deletes the letter $i$ from whatever position $j$ it appears in $\tau $ with coefficient $(-1)^{j-1}$.  It will suffice to prove that $d_{\underline{i}}(z_{\pi })=0$ for each $i\in \{ 1,2,\dots ,n\}$ to deduce $d(z_{\pi })=0$.

First we show that $d_{\underline{i}}(z_{\pi })=0$ for each letter $i$ that is either a left to right maximum or its immediate successor in $\pi $.

\begin{lem}\label{cancel-max-and-successor}
Let $\sigma$ be a pattern that starts with its largest letter. 
Let $\pi $ be a homology facet in $\I_n^{\sigma }$.
Let $j$ be a letter that is either a left-to-right maximum in $\pi $ or the immediate successor of some left-to-right maximum in $\pi $.  Then 
$d_{\underline{j}}(z_{\pi })=0$.
\end{lem}

\begin{proof}
Let $j$ be 
either a left-to-right maximum  in 
$\pi $, denoted $\pi_i$
or its immediate successor, denoted $\pi_{i+1}$.  For any injective word $\pi '$ appearing in $z_{\pi }$ with  $\pi_i$ to the immediate left (resp. right)  of   $\pi_{i+1}$, there is another injective word in $z_{\pi }$ having  the same sign as $\pi '$ in $z_{\pi }$ which is obtained from $\pi '$ by the adjacent transposition swapping the  values  $\pi_i$ and $\pi_{i+1}$.   Moreover, we showed in Lemma ~\ref{well-defined-shifted-permutation}  that all injective words in $z_{\pi }$ have these two letters, $\pi_i$ and $\pi_{i+1}$, in consecutive positions, implying all summands in $z_{\pi }$ are paired up in this manner.  
Applying 
$d_{\underline{i}}$ to these two injective words of length $n$  yields the same injective word of length $n-1$ in both cases, but with   opposite signs since in one case a letter is deleted from some position $l$ while in the other word the same letter is instead deleted from position $l+1$.  Since these two injective words had the same sign as summands in $z_{\pi }$, they cancel in $d_{\underline{i}}(z_{\pi })$, yielding the result by cancelling all of the terms in pairs in this manner.  
\end{proof}

Our pairing of summands to be cancelled in order to prove that $d_{\underline{j}}(z_{\pi })=0$ for homology facets $\pi $  for $j$ which is neither a left-to-right maximum nor its immediate successor in $\pi $ 
will  rely upon the next result.

\begin{lem}\label{stay-consistent}
If $S\subseteq U(\pi )$ 
satisfies $S = T \cup \{ i\} $ for some $i\not\in T$, then $\shift(S)[\pi ]$ may be obtained from $\shift(T)[\pi ]$ by shifting the letter $\pi_{i}$ leftward from its position in $\shift(T)[\pi ]$ by the minimal amount necessary to make it a left-to-right maximum. 
Thus, 
$$d_{\underline{\pi_i}}(\shift(T)[\pi ]) = \pm d_{\underline{\pi_i}}(\shift(S))[\pi ]).$$
\end{lem}

\begin{proof}
Our proof is by induction on $k$.  Let $T_k = T\cap \{ 1,\dots ,k\} $, and  let $S_k=S\cap \{ 1,\dots ,k\} $.  We will prove for each $k$ in turn that $d_{\underline{\pi_i}}(\shift(T_k)[\pi]) = \pm d_{\underline{\pi_i}}(\shift(S_k)[\pi ])$.
It suffices to prove  for each $\pi_k$ satisfying $k\in T$ and $k\ne i$ that $\pi_k$ appears in the same position in the unique  word $\gamma_k $ in the support of  $d_{\underline{\pi_i}}(\shift(T_k)(\pi ))$ as $\pi_k$ appears in the unique word $\gamma '_k$ in the support of  $d_{\underline{\pi_i}}(\shift(T_k)(\pi ))$, which  by induction on $k$ would imply that $\gamma_k = \gamma '_k$ for each $k$.  In other words, it suffices to check  for  each  letter $\pi_k$ satisfying   $k\in T$ and $k\ne i$ that applying $\shift(T)$ to $\pi = \pi_1\cdots \pi_n$ shifts $\pi_k$ leftward to the same position upon restriction to the subword in which $\pi_i$ is omitted that this letter $\pi_k$ shifts leftward into upon application of the map $\shift(S)$ to $\pi $ followed by restriction again  to the subword in which $\pi_i$ is omitted.   
We now turn to the proof of this claim.  Since the shifting processes are identical for $\shift(S)$ and $\shift(T)$ for $k<i$, the case of  $k<i$ follows vacuously.

By our inductive hypothesis,  we may assume that  the shifting of each of the letters $\pi_{k'}$ with $k'\in T$, $k'\ne i$  and $k'<k$ has the desired property in order to  verify this property for $\pi_k$. 
To prove our claim, we must handle the following two cases:
\begin{enumerate}
\item $\pi_k<\pi_i$ 
\item  $\pi_k>\pi_i$ 
\end{enumerate}

For  case (1), 
note that 
each letter in
$\{ \pi_1,\dots ,\pi_n\} \setminus \{ \pi_i,\pi_k\} $ that  $\pi_i$ would shift past when it is shifted leftward the minimal amount necessary in applying $\shift(S)$ to make it a left-to-right maximum is a letter that $\pi_k$ would also need to shift leftward past, regardless of whether $\pi_k$ is getting shifted leftward 
in the course of applying the operator  $\shift(S)$ or the operator $\shift(T)$; this follows directly from the fact that every letter which is larger than $\pi_i$ is also larger than $\pi_k$.  Moreover, we are ensured that $\pi_i$ does shift leftward past at least one letter, since $\pi_i$ is not initially a left-to-right maximum. 
Thus, $\pi_k$ must shift past all of the letters that are initially to the left of $\pi_i$ that are larger than $\pi_i$, of which there must be at least one such letter. If there are additional letters to the left of the position to which $\pi_i$ shifts that are larger than $\pi_k$ but not $\pi_i$, then $\pi_k$ also shifts to the left of these, regardless of whether applying $\shift(S)$ or $\shift(T)$.  The result follows in this case.

For case (2), the case with  $\pi_k>\pi_i$, we claim both for the application of the map $\shift(S)$ and for the map $\shift(T)$ that  $\pi_k$ will not shift leftward beyond  the position to which $\pi_i$ shifts in the application of $\shift(S)$; this holds regardless of whether $\pi_i$ has already been  shifted leftward or has been excluded from the shifting process,  since all letters to the left of the position to which $\pi_i$ shifts are smaller than $\pi_i$ and hence also smaller than $\pi_k$.  Also note that $\pi_k$ does not need to shift leftward past $\pi_i$ in either case  to be a left-to-right maximum since $\pi_k$ is larger than $\pi_i$.  The result follows. 
\end{proof}

\begin{ex}
Consider the permutation pattern $321$.  
Consider $\pi = \underline{61}2\underline{73}4\underline{85}$ with $S = \{ 3,6\} $ and $T = \{ 6\}$.  Then 
$\shift(T)[\pi ] = 4\underline{61}2\underline{73}\underline{85}$ and 
$\shift(S)[\pi ] = 24\underline{61}\underline{73}\underline{85}$.  In this case,
$\shift(S)[\pi ]$ is obtained from $\shift(T) [\pi ]$ by shifting the letter 2 leftward past the values  $6,1,$ and $2$.  Thus, 
$\shift(S)$ can be expressed as a composition of $\shift(T)$ and a permutation of length 3.  However, the length of the permutation $\shift(S)$ is only one more than the length of the permutation $\shift(T)$.  This discrepancy in length results from  the letters $2$ and $4$ getting  swapped and then unswapped when first applying $\shift(T)$ and then applying the leftward shift of the value $2$.  Fortunately, the sign of a permutation can be calculated using any expression for the permutation, reduced or otherwise; this is justified by Proposition~\ref{sign-tool}.  

Notice that applying the boundary map $d_{\underline{2}}$ to each of these two permutations, namely $\shift(S)(\pi )$ and $\shift(T)(\pi )$, yields the same result, up to a sign.  That is, 
$$d_{\underline{2}}(4\underline{61}2\underline{73}\underline{85}) = \pm d_{\underline{2}}(24\underline{61}\underline{73}\underline{85}). $$
This enables  cancellation of the terms 
$d_{\underline{2}}(46127385)$ and $d_{\underline{2}}(24617385)$ within $d_{\underline{2}}(\pi_z)$, provided that they appear with opposite sign in $d(z_{\pi }).$  Lemma~\ref{cancel-non-max} shows that such pairs indeed have opposite signs.
\end{ex}

To show that terms meant to cancel within $d_{\underline{i}}(z_{\pi })$ indeed have opposite signs and hence do cancel, we will make use of the following well-known fact.  This appears e.g. as Proposition 1.4.2, part 2, in \cite{BB} (where it is stated in the generality of any  Coxeter group):

\begin{prop}[Proposition 1.4.2, part 2, \cite{BB}]\label{sign-tool}
For all $u,w\in\S_n$, the length function $l$ on permutations satisfies $l(uw) \equiv l(u) + l(w)$ (mod 2).  Therefore $\sgn (uw) = \sgn(u) \sgn(w)$.
\end{prop}

This will allow us to record the signs of permutations that shift individual letters leftward and multiply these signs to get the sign of a permutation  which may be expressed as  a composition of one or more  such leftward  letter shifts.  In particular, we do not need to worry about the order in which these leftward shifting operations are carried out, even when some choices yield nonreduced expressions.
Specifically, we will use the following corollary of Proposition~\ref{sign-tool}:

\begin{cor}\label{cor:same-parity}
The parity of the number of letters $l$ appearing in a (reduced or nonreduced) expression $s_{i_1}\cdots s_{i_l}$ for a fixed  Coxeter group element $w$ is independent of the choice of expression.

Therefore,  if a set $S$ of letters is  shifted leftward via a reduced expression for a permutation $\shift(S)$, the length of this permutation agrees, up to a multiple of 2, with 
$l(u) + l(w)$ where $u$ shifts all but one of the letters   leftward, namely shifting the letters in some subset $T$ of $S$ leftward, thereby carrying out a permutation $\shift(T)$ as described above, and then $w$ shifts the unique letter in $S\setminus T$  leftward  to the unique position such that $\shift(S) = u\circ w$. 
\end{cor}

\begin{lem}\label{cancel-non-max}
Let $\sigma$ be a pattern that starts with its largest letter.  
Let $\pi $ be a homology facet in $\I_n^{\sigma }$. 
Let $i$ be a letter that is neither a left-to-right maximum in $\pi $ nor the immediate successor of any left-to-right maximum in $\pi $.  Then 
$d_{\underline{i}}(z_{\pi })=0$.
\end{lem}

\begin{proof}
Similarly to in Lemma~\ref{cancel-max-and-successor}, we will again cancel the summands of $z_{\pi }$ in pairs, but the proof will be a bit more involved in this case.

Let $i=\pi_j$ be a letter in $\pi $ that is neither  a left-to-right maximum nor its immediate successor in our given  homology facet $\pi $.  Consider any fixed $M\subseteq L(\pi )$.  
For each $S, T\subseteq U(\pi )$ satisfying  $ T = S\cup \{ j\} $ with 
$j\not\in S$, we pair the  injective word 
$$\pi (T): = (\prod_{m\in M}(\pi_m,\pi_{m+1}))\shift(T)[\pi ]$$ appearing with nonzero coefficient in $z_{\pi }$ 
with the  injective word 
$$\pi (S)  \coloneqq  (\prod_{m\in M}(\pi_m,\pi_{m+1}))\shift(S)[\pi ]$$ also appearing with nonzero coefficient in $z_{\pi }$.  Recall our nonstandard convention 
that the shift operators specify a list of positions and then shift leftward the letters appearing in those positions whereas the operator  $\prod_{m\in M}(\pi_m,\pi_{m+1})$ 
swaps the pairs of values $(\pi_m,\pi_{m+1})$ that appear  in positions $m,m+1$ in $\pi $ itself.

By Lemma~\ref{stay-consistent}, $\shift(T)[\pi ]$ is obtained from $\shift(S)[\pi ]$ by shifting $\pi_j$ leftward the minimal amount necessary to make $\pi_j $ a left-to-right maximum.  Thus, we may use Corollary~\ref{cor:same-parity} to deduce that  $\sgn(\shift(T)) = (-1)^{s-r}\sgn (\shift(S)$ where $r$ is the position of $\pi_j$ in $\shift(T)[\pi ] $
and $s$ is the position of $\pi_j$ in $\shift(S)[\pi ]$.
Note also that applying $\prod_{m\in M}(\pi_m,\pi_{m+1})$ to $\shift(S)[\pi ]$ (resp. $\shift(T)[\pi ])$ does not impact the location of $\pi_j$ since $\pi_j\in U(\pi )$. 
Thus, we may deduce
$$\sgn (\pi (T)) = (-1)^{s-r}\sgn (\pi (S)).$$

Also note that $(-1)^{|S|} = -(-1)^{|T|}$ since $|S|=|T|+1$.  These calculations  combine to show that the coefficient of $\pi (S)$ in $z_{\pi }$ differs from the coefficient of $\pi (T)$ in $z_{\pi }$ by a factor of $(-1)^{s-r+1}$.  On the other hand, the coefficients introduced by applying the boundary map $d_{\underline{i}}$ to $\pi (S)$ and $\pi (T)$, respectively,  differ by a factor of $(-1)^{s-r}$, due to $i$ being deleted from position $s$ in $\pi (S)$ and from  position $r$ in $\pi (T)$.  Therefore the coefficients of 
$d_{\underline{i}}(\pi (S))$ and $d_{\underline{i}}(\pi (T))$ within $d_{\underline{i}}(z_{\pi })$ differ by a factor of $(-1)^{s-r+1}\cdot (-1)^{s-r} = -1$, as desired.
\end{proof}

\begin{ex}
Consider once again   
the homology facet   
$\pi = \underline{61}2\underline{73}4\underline{85}$ in $\I_8^{321}$.  Again consider the subsets $S = \{ 3,6\} $ and $T = \{ 6\}$ of $U(\pi )$.  The letter getting shifted leftward by $\shift(S)$ that does not get shifted leftward by $\shift(T)$ is the letter 2.

Note that 
$\shift(T)[\pi ] = 4\underline{61}2\underline{73}\underline{85}$ has the letter $2$ in position $s=4$ while  
$\shift(S)[\pi ] = 24\underline{61}\underline{73}\underline{85}$ has
the letter $2$ in position $r=1$, so $s-r=3$. 

Since  $|S|=2$ and $|T|=1$, 
the coefficients for the summands   
$\prod_{m\in M}(\pi_m,\pi_{m+1})\shift(T)$  and  
$\prod_{m\in M}(\pi_m,\pi_{m+1})\shift(S)$ in $z_{\pi }$ differ by a factor of $(-1)^{s-r}\cdot (-1)^{|S|-|T|} = (-1)^3(-1)^1 = 1$.  
Although these summands  
have the same sign within $z_{\pi }$,  
applying the boundary map $d_{\underline{2}}$ to each of these summands gives results that differ  
by a sign of 
$(-1)^{s-r}=(-1)^3 = -1$.
The net result is opposite signs in $d(z_{\pi })$, just as needed to ensure these two summands will cancel in $d(z_{\pi })$. 
\end{ex}

Combining Lemmas~\ref{cancel-max-and-successor} and~\ref{cancel-non-max} above yields:

\begin{prop}
Let $\sigma$ be a pattern that starts with its largest letter, and let $\pi $ be a homology facet  in our shelling for $\I_n^{\sigma }$ in which we order the $\sigma $-avoiding permutations written in one line notation lexicographically. Then the corresponding element $z_{\pi }$ of the top chain group satisfies $d(z_{\pi })=0$, implying $[z_{\pi }]$ is a homology class.
\end{prop}

The next lemma will facilitate  a triangularity argument to show that the set of  homology classes $[z_{\pi }]$ given by the homology facets  $\pi  $ of our shelling  comprise a homology basis.

\begin{lem}\label{get-smaller}
Let $\sigma$ be a pattern that starts with its largest letter, and let $\pi $ be a homology facet in our shelling for $\I_n^{\sigma }$ in which we order the $\sigma $-avoiding permutations written in one-line notation lexicographically.
Then each permutation appearing in $z_{\pi }$ with nonzero coefficient is at least as small in lexicographic order  as $\pi $.      
\end{lem}

\begin{proof}
First note for any permutation $\tau = \tau_1\cdots \tau_n$ that shifting a letter $\tau_j$ that is not a left-to-right maximum in $\tau $ leftward to the immediate left of a letter $\tau_i$ that is larger than $\tau_j$ yields a permutation that is lexicographically smaller than $\tau $.  This implies for each $T\subseteq U(\pi )$ that $\shift(T)[\pi ]$ is lexicographically at least as small as  $\pi$, due to being obtained by a series of leftward shifts of individual letters $\tau_j$ that each in turn produce a lexicographically smaller permutation.  

Next we show for any $M\subseteq L(\pi )$ and any $T\subseteq U(\pi )$ that 
$\prod_{m\in M}(\pi_m,\pi_{m+1})(\shift(T)[\pi ])$ is at least as small lexicographically as $\shift(T)[\pi ]$.  To check this, we use the fact (justified in Observation~\ref{stay-successor})  that $\pi_m$ is to the immediate left of $\pi_{m+1}$ in $\shift(T)(\pi )$ and satisfies $\pi_m>\pi_{m+1}$ for each $m\in M$.  From this it follows that each of the adjacent transpositions $(\pi_m,\pi_{m+1})$ for $m\in M$ produces a lexicographically permutation that the one to which it was applied.  Since each permutation appearing in $z_{\pi }$ is obtained as some $(\prod_{m\in M}(\pi_m,\pi_{m+1}))(\shift(T)[\pi ])$, the result follows.
\end{proof}

\begin{rk} Proposition 1.1 of \cite{Wachs-cohomology} does not apply directly to give the homology basis result we prove below in~\ref{homology-basis-theorem}, but it does serve as inspiration for this next result.  The issue preventing us from direction applying Proposition 1.1 of \cite{Wachs-cohomology} is that in some cases we have homology facets appearing in the support of the basis elements given by other homology facets.
\end{rk}

Finally we deduce:

\begin{thm}\label{homology-basis-theorem}
For any pattern $\sigma$ that starts with its largest letter,
the homology classes $[z_{\pi }]$ indexed by the homology facets $\pi $ in the shelling for $\I_n^{\sigma }$ which orders permutations lexicographically  
comprise a homology basis for the top homology of $\I_n^{\sigma }$.
\end{thm}

\begin{proof}
This proof proceeds by a triangularity argument.
Here we use the fact that  each permutation $\pi '$ appearing with nonzero coefficient in $z_{\pi }$ for $\pi $ a homology facet 
is  less than or equal to $\pi $ in lexicographic order, by 
Lemma~\ref{get-smaller}.  This allows us to
show that the following matrix is upper triangular so as to deduce that the set of homology class representatives $z_{\pi }$ are linearly independent.  The rows and columns of this matrix are  
indexed by the permutations in $\I_n^{\sigma }$, in each case ordered lexicographically.  The row indexed by a  homology facet $\pi $ has entry $c(\pi ,\pi ')$  in each  column indexed by a permutation $\pi '$ appearing in 
$z_{\pi }$ with nonzero coefficient $c(\pi ,\pi')$  and has a 0 in every other column.  Each row indexed by a permutation that is not a homology facet consists entirely of 0's.   Consider now the submatrix consisting of the nonzero rows.   The row  indexed by $\pi $ in this submatrix has a 1 in the column also indexed by $\pi $, due to $\pi $ appearing in $z_{\pi }$ with coefficient 1.  This row has only 0's to the right of this column indexed by $\pi $, by Lemma~\ref{get-smaller}.  Thus, the rank of this matrix equals its number of rows.  This shows that the elements of our proposed homology basis are linearly independent.  But they are equinumerous with the rank of the top homology group, since that equals the number of homology facets.  Therefore  this collection of cycles is a homology basis.
\end{proof}

\subsection{Homology basis for $\I_n^{\sigma }$ for extremal-endpoint patterns}

Using standard isomorphisms discussed earlier in the paper, it is quite easy to generalize Theorem~\ref{homology-basis-theorem} 
as follows:

\begin{thm}\label{homology-theorem-two}
Given any permutation pattern $\sigma $ whose largest or smallest  letter is leftmost or rightmost, in other words, any extremal-endpoint pattern,  the complex $\I_n^{\sigma }$ admits a homology basis indexed by the homology facets of the shelling for $\I_n^{\sigma }$ given in Corollary~\ref{cor:extremal-endpoint}. 
\end{thm}

\begin{proof}
To handle permutation patterns whose largest letter is rightmost, we simply apply the isomorphism of complexes sending each injective word $\pi_1\cdots \pi_r$ to its reversal $\pi_r\cdots \pi_1$.  To handle permutation patterns whose smallest letter is leftmost, apply the isomorphism of complexes sending each injective word $\pi_1\cdots \pi_r$ to the injective word $(n+1-\pi_1)\cdots (n+1-\pi_r)$ where each letter $\pi_i$ is replaced by $n+1-\pi_i$.  For permutation patterns whose smallest letter is rightmost, apply the composition of these two isomorphisms of complexes.   

Thus, we may push our homology basis construction for patterns whose largest letter is leftmost  through these isomorphisms to get homology bases for $\I_n^{\sigma }$ for any pattern $\sigma $ (of any length) whose largest or smallest letter is leftmost or rightmost.
\end{proof}

This implies: 

\begin{cor}
Given any pattern $\sigma $ of length 3 or 4 not in the set $\{ 3142, 3412, 2143, 2413\} $, the 
complex $\I_n^{\sigma }$ admits a homology basis indexed by the homology facets of our shelling for $\I_n^{\sigma }$.
\end{cor}

\begin{proof}
First we note that the patterns  $321$ and $312$ each start with their largest letter, allowing us to use the results of the previous section, specifically Theorem~\ref{homology-basis-theorem}, to obtain a homology basis in either of these cases.   Next observe that the patterns 
$123$ and $213$ each end with their largest letter, allowing us to use Theorem~\ref{homology-theorem-two} to handle these cases.  The patterns $132$ and $231$ start (resp.\ end) with their smallest letter, so again Theorem~\ref{homology-theorem-two} covers these cases.  
This completes the proof for patterns of length 3.  

The proof for patterns of length 4 is similar.  The only patterns not covered by Theorem~\ref{homology-basis-theorem} or Theorem~\ref{homology-theorem-two} are those patterns of length 4 whose middle letters are 1 and 4, namely the patterns $2143$, $2413$, $3142$ and $3412$.  
\end{proof}

\section{Open questions}\label{sec:open}

There are many open questions regarding the complexes of injective words avoiding a pattern (or collection of patterns).  We collect some of the main ones below and also describe  data and partial results we have obtained supporting affirmative answers to some of these questions.  

We begin with four of our main questions, all gathered together within Question~\ref{qn:main}.

\begin{qn}\label{qn:main}
    For each of the following four conditions, determine for which patterns $\sigma$ the complex $\I_n^{\sigma }$ satisfies the condition for all $n$:
    \begin{enumerate}
       \item $\I_n^{\sigma }$ is pure of dimension $n-1$.
        \item  $\I_n^{\sigma }$ is shellable.
          \item The homology of $\I_n^{\sigma }$ is concentrated in top degree.
        \item The $h$-numbers of $\I_n^{\sigma }$ are all nonnegative.
    \end{enumerate}
\end{qn}

As discussed in Section~\ref{sec:shellability}, if a complex is pure and shellable, then its homology is concentrated in top degree and all its $h$-numbers are nonnegative. Thus, any pattern $\sigma$ satisfying conditions (1) and (2) will also satisfy (3) and (4). 
It is conceivable that all four conditions hold for every pattern $\sigma$ and every $n$, in the sense that we have not found any counterexamples.

We know by Corollary~\ref{cor:first-pattern-group} and Proposition~\ref{prop:second-pattern-group} that 
$\I_n^{\sigma }$ is pure of dimension $n-1$ for large classes of patterns $\sigma$, including any pattern of length at most 4, as stated in Corollary~\ref{cor:three-and-four-pure}. Additionally, we know by
Corollaries~\ref{cor:length3} and~\ref{cor:shellable-length4} 
that $\I_n^{\sigma }$ is shellable for any pattern of length at most 4 other than possibly $3142$ and $2413$.
We think there is a reasonable chance that the complexes given by these  two patterns
are shellable as well; it suffices to prove this for either one of them, since they give rise to  isomorphic complexes by virtue of either  pattern being the reversal of the other pattern.  This leads to the following question which we list separately from Question~\ref{qn:main} 
for the sake of emphasis.

\begin{qn}\label{3142}    
Is the complex $\I_n^{3142}$ (and hence $\I_n^{\sigma }$ for each  pattern $\sigma $ of length at most 4) shellable for every $n$?
\end{qn}

We have computationally verified that Question~\ref{3142} has an affirmative answer for all $n\le10$, by checking that the lexicographic order of the facets of $\I_n^{3142}$ is a shelling order.
We also have a proof, in analogy to Lemma~\ref{lem:3412blockers-sufficient}, that for any $\pi\in\S_n(3142)$ for which $\pi_i>\pi_{i+1}$, both of these entries $\pi_i$ and $\pi_{i+1}$ are $3412$-blockers of $\pi$ with respect to $\lex$.

We have also checked computationally that $\I^{\sigma}_n$ is shellable for every pattern $\sigma$ of length at most $5$ and every $n\le 9$.
In fact, in each such case, we have found that there is some $\tau\in\{\sigma,\sigma^R,\sigma^C,\sigma^{RC}\}$ for which the complex $\I_n^\tau$ (which we know is isomorphic to $\I_n^\sigma$ by Propositions~\ref{isomorph-one} and~\ref{isomorph-two})
is shellable with respect to $\lex$ (for $n\le9$). 
It is not clear, however, what conditions on $\tau$ enable this property. 
For example, of the patterns in the symmetry class $\{42513,31524,24153,35142\}$, only $\tau=35142$ has the property that $\I_9^{\tau}$ is shellable with respect to $\lex$; 
and of the patterns in the symmetry class $\{41352,25314\}$, only $\tau=25314$ has this property.

Note that if $\sigma$ and $\tau$ are Wilf-equivalent patterns, then the complexes 
$\I_n^{\sigma }$ and $\I_n^{\tau}$ have the same $h$-numbers as each other. In particular, since $3142$ and $2413$ are Wilf-equivalent to $4132$, and we know that $\I^{4132}_n$ is shellable by Theorem~\ref{thm:shellability_start_with_k}, we deduce that condition (4) in Question~\ref{qn:main} holds for every pattern of length up to 4. Note, however, that we cannot deduce shellability of $\I^{3142}_n$, since it is possible for complexes $\I^\sigma_n$ and $\I^\tau_n$ to be non-isomorphic even if $\sigma$ and $\tau$ are 
Wilf-equivalent, as shown in Proposition~\ref{prop:not_isomorphic}.

In light of our  homology basis construction, another natural question is the following. 

\begin{qn}
Are there homology bases analogous to those constructed in Section~\ref{homology-section} for the complexes of 3412-avoiding (resp.\ separable) injective words, with basis elements indexed by the homology facets of the shelling we have given for the complex of 3412-avoiding (resp.\ separable) injective words?
\end{qn}

One may also generalize Question~\ref{qn:main} to collections $\Sigma$ of patterns.  However, in that case one would likely first want to focus on part (1), namely the following question.

\begin{qn}
Which collections $\Sigma$ of patterns yield complexes $\I_n^\Sigma$ all of whose maximal faces are permutations in $\S_n$?  In other words, for which $\Sigma$ is $\I_n^{\Sigma}$ pure of dimension $n-1$?
\end{qn}

There exist  collections $\Sigma$ of patterns such that not all of the maximal faces in $\I_n^{\Sigma}$ are permutations.  For example, if $\Sigma$ consists of all possible patterns of length 3 then none of the faces in $\I_n^{\Sigma}$ can be greater than 1-dimensional regardless of how large $n$ is.

We end with a more open-ended question.  
Given how the derangement numbers arise from the full complexes of injective words as the rank of the top homology group  
and the Riordan numbers likewise come up in studying subcomplexes $\I_n^{\sigma }$ for any pattern $\sigma $ of length 3, it is natural also to ask the following question.

\begin{qn}
What other patterns $\sigma $ (or collections $\Sigma$ of patterns) yield pure complexes $\I_n^{\sigma }$ (resp.\ $\I_n^{\Sigma}$) with homology concentrated in top degree  for which the rank of the top homology group (or equivalently, the absolute value of the reduced Euler characteristic) is a combinatorially interesting sequence of numbers indexed by $n$?  In such cases, does one obtain any  interesting identities  by comparing the reduced Euler characteristic calculated  as an alternating sum of Betti numbers  with  the reduced Euler characteristic expressions  for these same  complexes  as alternating sums of face numbers?  
\end{qn}

\section*{AI Statement}

AI was not used in any manner in obtaining our mathematical results or in our exposition of the results.

\section*{Acknowledgements}
The first author was partially supported by Simons Collaboration Grant \#929653.   The third author was partially supported by Simons Travel Funding for Mathematicians.

\end{document}